\documentclass[]{svmono}

\makeatletter
\@addtoreset{equation}{section}
\makeatother

\usepackage{multicol}        
\usepackage[bottom]{footmisc}

\usepackage{color}

\newcommand{\blue}{\color{blue}}

\renewcommand{\thesection}{\arabic{section}}

\usepackage{amssymb}
\usepackage{latexsym}
\usepackage{amsmath}
\usepackage{amsfonts}
\usepackage{amstext}
\usepackage{euscript}
\usepackage{mathrsfs}

\usepackage{txfonts}

\usepackage{dsfont}

\newcommand\gotX{{\mathfrak{X}}}

\newcommand\dR{{\mathbb{R}}}
\newcommand\dC{{\mathbb{C}}}

\newcommand\dN{{\mathbb{N}}}

\newcommand{\gb}{{\beta}}
\newcommand{\gd}{{\delta}}
\newcommand{\gD}{{\Delta}}

\newcommand{\go}{{\omega}}

\newcommand\gt{{\tau}}

\newcommand{\gT}{{\Theta}}

\newcommand{\e}{\varepsilon}

\newcommand\cA{{\mathcal{A}}}
\newcommand\cB{{\mathcal{B}}}
\newcommand\cC{{\mathcal{C}}}

\newcommand\cG{{\mathcal{G}}}

\newcommand\cI{{\mathcal{I}}}
\newcommand\cK{{\mathcal{K}}}
\newcommand\cL{{\mathcal{L}}}

\newcommand\cO{{\mathcal{O}}}

\newcommand\cU{{\mathcal{U}}}

\newcommand\1{{\mathds{1}}}

\newtheorem{hypotheses}[theorem]{Hypotheses}

\newcommand{\ba}{\begin{array}}
\newcommand{\ea}{\end{array}}
\newcommand{\bea}{\begin{eqnarray}}
\newcommand{\eea}{\end{eqnarray}}
\newcommand{\bead}{\begin{eqnarray*}}
\newcommand{\eead}{\end{eqnarray*}}
\newcommand{\be}{\begin{equation}}
\newcommand{\ee}{\end{equation}}
\newcommand{\bed}{\begin{displaymath}}
\newcommand{\eed}{\end{displaymath}}

\newcommand{\bt}{\begin{theorem}}
\newcommand{\et}{\end{theorem}}
\newcommand{\bl}{\begin{lemma}}
\newcommand{\el}{\end{lemma}}
\newcommand{\bc}{\begin{corollary}}
\newcommand{\ec}{\end{corollary}}
\newcommand{\br}{\begin{remark}}
\newcommand{\er}{\end{remark}}
\newcommand{\bd}{\begin{definition}}
\newcommand{\ed}{\end{definition}}
\newcommand{\bprop}{\begin{proposition}}
\newcommand{\eprop}{\end{proposition}}
\newcommand{\bh}{\begin{hypotheses}}
\newcommand{\eh}{\end{hypotheses}}
\newcommand{\bal}{\begin{aligned}}
\newcommand{\eal}{\end{aligned}}

\newcommand{\vertiii}[1]{{\left\vert\kern-0.25ex\left\vert\kern-0.25ex\left\vert #1
    \right\vert\kern-0.25ex\right\vert\kern-0.25ex\right\vert}}

\def\e{{\rm e\,}}
\def\dom{{\rm dom\,}}

\def\dom{{\rm dom\,}}

\DeclareMathOperator*{\esssup}{ess\,sup}

\spnewtheorem{theorem}{Theorem}[section]{\bfseries}{\itshape}
\spnewtheorem{lemma}[theorem]{Lemma}{\bfseries}{\itshape}
\spnewtheorem{definition}[theorem]{Definition}{\bfseries}{\itshape}
\spnewtheorem{proposition}[theorem]{Proposition}{\bfseries}{\itshape}
\spnewtheorem{corollary}[theorem]{Corollary}{\bfseries}{\itshape}
\spnewtheorem{example}[theorem]{Example}{\bfseries}{\itshape}
\spnewtheorem{remark}[theorem]{Remark}{\bfseries}{\itshape}
	
\newcommand{\cl}{\mathcal}

\newcommand{\C}{\mathbb{C}}
\begin{document}

\title{Operator-norm Trotter product formula on Banach spaces}

\author{Valentin A. Zagrebnov \\
CNRS-Universit\'{e} d'Aix-Marseille\\
Institut de Math\'{e}matiques de Marseille \\
CMI - Technop\^{o}le Ch\^{a}teau-Gombert \\
39, rue F. Joliot Curie, 13453 Marseille Cedex 13, France}



\hspace{1cm}\textbf{Operator-norm Trotter product formula on Banach spaces}

\vspace{0.5cm}

\hspace{3cm}\textbf{Valentin A. Zagrebnov}

\vspace{0.2cm}

\hspace{2cm} CNRS-Universit\'{e} d'Aix-Marseille

\hspace{2cm} Institut de Math\'{e}matiques de Marseille

\hspace{2cm}  CMI - Technop\^{o}le Ch\^{a}teau-Gombert

\hspace{2cm} 39, rue F. Joliot Curie

\hspace{2cm} 13453 Marseille Cedex 13, France

\vspace{1cm}

\hspace* {3.2cm} \textit{In memory of Academician Vasili{\v{\i}} Sergeevich Vladimirov} \\
\hspace* {6cm} \textit{on the 100$^{\rm th}$ anniversary of his birthday}

\section*{\blue{Abstract}}
In this paper we collect results concerning the \textit{operator-norm} convergent
\textit{Trotter} product formula for two semigroups $\{\e^{- t A}\}_{t\geq 0}$,
$\{\e^{- t B}\}_{t\geq 0}$, with densely defined generators $A$ and $B$ in a \textit{Banach} space.
Although the \textit{strong} convergence in Banach space for contraction semigroups is known
since the seminal paper by Trotter (1959),
which after more than three decades was extended to convergence in the \textit{operator-norm}
topology in \textit{Hilbert} spaces by Rogava (1993),
the {operator-norm} convergence in a \textit{Banach} space was established {{only in}} (2001).

For the first time this result was established under hypothesis that one of the involved
into the product formula contraction semigroups, e.g. $\{\e^{- t A}\}_{t\geq 0}\, $, is
\textit{holomorphic} together with certain conditions of \textit{smallness} on generators $B$ and
$B^*$ with respect to generators $A$ and $A^*$.
Note that in spite of a quite strong assumptions on operators $A$ and $B$ the proof of the
operator-norm convergent Trotter product formula on a Banach space is (unexpectedly) involved
and technical.

To elucidate the question of how far these conditions are from optimal ones we
show an Example of the \textit{operator-norm} convergent {Trotter}
product formula for two semigroups $\{\e^{- t A}\}_{t\geq 0}$ and $\{\e^{- t B}\}_{t\geq 0}$ on a
Banach space, where hypothesis on operator $A$ is relaxed to condition that $A$ is generator of
a \textit{contraction} semigroup.


\section{\blue{Preliminaries}}\label{sec:8.1}
\subsection{\blue{Bounded semigroups on $\mathfrak{X}$ }}\label{subsec:8.1.1}
For what follows the properties of holomorphic (contraction) semigroups on a Banach space
$\mathfrak{X}$ are essential. Therefore, we start by a suitable for our aim recall of details
concerning the bounded,
holomorphic semigroups, and fractional powers of their generators. We begin with definitions and
properties to introduce certain notations adapted in this section for semigroups on $\mathfrak{X}$.
\begin{definition}\label{def:8.1.1}
{\rm{We would remind that a family $\{U(t)\}_{t\geq 0}$ of bounded linear operators on a Banach
space $\mathfrak{X}$ is called a one-parameter strongly continuous semigroup if it satisfies the
conditions:\\
\begin{tabular}{rl}
(i) &$U(0)={1}$, \\
(ii) &$U(s+t) = U(s)U(t)$ for all $s,t\geq 0$, \\
(iii) &$\lim_{t\rightarrow +0} U(t)x = x$ for all $x\in{\mathfrak{X}}$.
\end{tabular}
}}
\end{definition}

We recall some immediate consequences of this definition:

\begin{itemize}
   \item There are constants $C_A\geq 1$ and $\gamma_A\in{\dR}$, depending on the generator of the
   semigroup,
   such that
   $\|U(t)\| \leq C_A \e^{\gamma_A t}$ for all $t\geq 0$.
   \item $t\mapsto U(t)$ is a strongly continuous function from $[0,+\infty)$ onto the algebra
   $\mathcal{L}(\mathfrak{X})$ of bounded linear operators on $\mathfrak{X}$.
   \item There exists a closed densely defined linear operator $A$ on $\mathfrak{X}$ with domain
   ${\rm{dom}}(A)$, called the generator
   of the semigroup, such that $\lim_{t\rightarrow +0}(U(t)x-x)/t = -Ax$ for any $x\in{\rm{dom}}(A)$,
   that is, by convention $U(t):=\e^{-tA}$.
   \item The resolvent of the generator satisfies the estimate
   $\|R_{A}(-\lambda)\|$ $=$  $\|(A+\lambda \mathds{1})^{-1}\|$ $\leq$
   $C_A / (\mbox{Re}(\lambda) - \gamma_A)$ for
  all $\lambda$ with
   $\mbox{Re}(\lambda)>\gamma_A$, thus the open half plane with $\mbox{Re}(z) < -\gamma_A$ is
   contained into the resolvent set of $A$, which is
   defined as $\rho(A) = \{z\in{\dC} : \|R_{A}(z)\| < +\infty \}$.
\item If $\gamma_A\leq 0$, $U(t)$ is called a bounded semigroup (otherwise, $U(t)$
is called a quasi-bounded semigroup of type $\gamma_A>0$). For any strongly continuous
semigroup, we can construct a bounded semigroup by adding some constant
$\eta\geq\gamma_A$ to its generator. Let $\tilde{A} = A+\eta \mathds{1}$, then for the semigroup
$\tilde{Q}(t)$ generated by $\tilde{A}$, one has $\|\tilde{Q}(t)\|\leq C_A$, $t\geq 0$,
and the open half-plane $\mbox{Re}(\lambda)<\eta-\gamma_A$ is included into the
resolvent set $\rho(\tilde{A})$ of $\tilde{A}$. So it is not restrictive to suppose
that the considered semigroup $U(t)$ is bounded and that the set $\{z\in{\dC}: \mbox{ Re}(z)\leq 0\}
\subseteq\rho(A)$.
\item If $\|U(t)\|\leq 1$, $t\geq 0$, the semigroup is called a contraction semigroup.
We comment that the method of the preceding remark does not permit to construct a
contraction semigroup from a bounded semigroup in general, since it can not change the value of the
constant $C_A$.
\end{itemize}

Below we need a characterization of generators of these contraction semigroups. First we
recall that the space of linear bounded functionals
${\mathfrak{X}}^* = \mathcal{L}(\mathfrak{X}, \mathbb{C})$ is a \textit{dual} of the
Banach space $\mathfrak{X}$ and that ${\mathfrak{X}}^*$ is itself a Banach space.
Recall that a linear operator $A$ in $\mathfrak{X}$ is accretive if for all pairs
$(u, \phi)\in{\rm{dom}}(A)\times{\mathfrak{X}}^*$ with $\|u\|_{\mathfrak{X}}=1,\
\|\phi\|_{\mathfrak{X}^*}=1, \ ( u, \phi ) =1$, one has $\mbox{Re} \, ( Au,\phi ) \geq 0$.
We also add that
a densely defined in $\mathfrak{X}$ accretive operator $A$ is generator of contraction semigroup
if the range of $\lambda \mathds{1} + A$ is $\mathfrak{X}$ for some $\lambda>0$,

Now we prove a series of estimates indispensable throughout this paper.
\bl\label{lem:8.1.2}
Let $U(t)$ be a bounded semigroup with boundedly invertible generator $A$, then for all $t\geq 0$,
and for any $n\in\dN$, we have:
\begin{equation}\label{eq:8.1.1}
\left(U(t)-\sum_{k=0}^{n}{\frac{(-tA)^k}{k!}}\right)A^{-n-1} = -\int_0^t d\tau \, \left(U(\tau)-
\sum_{k=0}^{n-1}{\frac{(-\tau A)^k}{k!}}\right) A^{-n} \, ,
\end{equation}
\begin{equation}\label{eq:8.1.2}
\left\|\, \left(U(t)-\sum_{k=0}^n{\frac{(-tA)^k}{k!}}\right)A^{-n-1}\right\|\ \leq\
C_A {\frac{t^{n+1}}{(n+1)!}}.
\end{equation}
\el
\begin{proof}
We proceed by induction, and we first prove that
\begin{equation}\label{eq:8.1.3}
(U(t)-\mathds{1})x = -\int_0^t d\tau \, U(\tau)\ A \, x \ , \ \  x\in{\rm{dom}}(A).
\end{equation}
Note that for any $\epsilon>0$ the semigroup properties yields the representation
\begin{eqnarray*}
\int_0^t ds \, U(s)\, {U(\epsilon)-\frac{\mathds{1}}{\epsilon}} & = &
\int_0^t ds \, {U(s+\epsilon) - \frac{U(s)}{\epsilon}} \\
& = & \int_t^{t+\epsilon}ds \, {\frac{U(s)}{\epsilon}} -
\int_0^\epsilon ds \, {\frac{U(s)}{\epsilon}} \\
& = & (U(t)-\mathds{1}) \, {\frac{1}{\epsilon}} \int_0^\epsilon ds \, U(s) \, .
\end{eqnarray*}
Moreover, one also gets:
\begin{equation*}
 \lim_{\epsilon\rightarrow 0} {\frac{1}{\epsilon}} \int_0^\epsilon  ds \ U(s)\ x  = x \ , \ \ \
 x\in{\mathfrak{X}} \ ,
\end{equation*}
\begin{equation*}
\lim_{\epsilon\rightarrow 0} \frac{U(\epsilon)-\mathds{1}}{\epsilon} \, x = - \, A \,x \ , \ \ \
 x\in{\rm{dom}}(A)\,.
\end{equation*}
This proves (\ref{eq:8.1.3}), and since $A$ is boundedly invertible, we obtain (\ref{eq:8.1.1})
for $n=0$.
Furthermore,
since $U(t)$ is bounded by $C_A$, we obtain the estimate (\ref{eq:8.1.2}) for $n=0$.

Suppose that (\ref{eq:8.1.1}) and (\ref{eq:8.1.2}) are true for some $n$, then a
straightforward calculation
leads to
(\ref{eq:8.1.1}) for $n+1$.
Hence, using the representation (\ref{eq:8.1.1}) and (\ref{eq:8.1.2}) for $n$ to estimate the
integrand, we obtain (\ref{eq:8.1.2}) for $n+1$, which completes the proof by induction.
\hfill $\square$
\end{proof}

Similarly, we obtain a representation for a restricted development of $(\mathds{1} + A)^{-1}$.
\bl\label{lem:8.1.3}
Let $A$ be as in Lemma \emph{\ref{lem:8.1.2}}. Then for any $n\geq 0$ :
\begin{equation}\label{eq:8.1.4}
(\mathds{1} + A)^{-1}A^{-n-1} = \left(\sum_{k=0}^n (-A)^k \right)A^{-n-1} +
(-1)^{n+1}(\mathds{1}+ A)^{-1}.
\end{equation}
\el
\begin{proof}
For $n=0$, the representation (\ref{eq:8.1.4}) follows from the resolvent formula:
\begin{equation}\label{eq:8.1.5}
(\mathds{1} + A)^{-1} - A^{-1} = -(\mathds{1} + A)^{-1}A^{-1}.
\end{equation}
Suppose that (\ref{eq:8.1.4}) holds for an integer $n > 1$, then:
\begin{equation}\label{eq:8.1.6}
(\mathds{1} + A)^{-1}A^{-n-2} = \left(\sum_{k=0}^n (-A)^k \right)A^{-n-2} +
(-1)^{n+1}(\mathds{1} + A)^{-1}A^{-1}.
\end{equation}
Applying (\ref{eq:8.1.5}) to the last term of (\ref{eq:8.1.6}) we get the representation
(\ref{eq:8.1.4}) for
$n+1$,
and thus for any $n$ by induction.
\hfill $\square$
\end{proof}

\bl\label{lem:8.1.4}
If $U(t)$ is a bounded semigroup with boundedly invertible generator $A$ then :
\begin{equation}\label{eq:8.1.7}
\left\|{\frac{1}{t^2}}\left((\mathds{1}+t\, A)^{-1}-U(t)\right)A^{-2}\right\|\ \leq\ 3C_A /2 \ ,
\ \quad  t>0.
\end{equation}
\el
\begin{proof}
By Lemma \ref{lem:8.1.2} one gets
\begin{equation}\label{eq:8.1.8}
\|(U(t)-\mathds{1}+ t \, A)\ \frac{1}{t^2} \, A^{-2}\|\ \leq\ {\frac{C_A }{2}} \ .
\end{equation}
On the other hand by Lemma \ref{lem:8.1.3}, we have
\begin{equation}\label{eq:8.1.9}
\left\|\left((\mathds{1} + t\,A)^{-1} - \mathds{1} +
t\,A \right) \ \frac{1}{t^2} \, A^{-2}\right\|\ = \ \|(\mathds{1} + t\,A)^{-1}\|\ \leq\ C_A \ .
\end{equation}
Here the last estimate follows from $(\mathds{1} + t\,A)^{-1} = (1/t) \, R_{A}(-1/t)$ and
$\|R_{A}(-\lambda)\|\leq C_A /(\lambda+\delta)$,
$\delta\geq 0$, which is valid for bounded semigroups with boundedly invertible generators.
Hence (\ref{eq:8.1.7}) follows from
(\ref{eq:8.1.8}) and (\ref{eq:8.1.9}).
\hfill $\square$
\end{proof}

\subsection{\blue{Holomorphic contraction semigroups on $\mathfrak{X} \ \ \ $}}\label{subsec:8.1.2}

Now let $U: z \mapsto U(z)$ be a family of operators with $z$ taking their values in the sector
of the complex plane:
\begin{equation}\label{sector-omega}
S_\theta = \left\{ z\in{\dC}:\ z\neq 0 \mbox{ and } |\arg(z)|<\theta\right\}
\end{equation}
where $0<\theta\leq\pi/2$.
\bd\label{def:8.1.5}
{\rm{Recall that the family of operators $\{U(z)\}_{z\in S_\theta}$ is a bounded holomorphic
semigroup of semi-angle $\theta\in (0, \pi/2]$ on a Banach space $\mathfrak{X}$ if it satisfies
the following conditions:}}\\
{\rm{
\begin{tabular}{rl}
(i) & If $0<\epsilon<\theta$, then  $\|U(z)\|\leq M_\epsilon$ for all $z\in S_{\theta-\epsilon}$ and
some $M_\epsilon<\infty$. \\
(ii) & $U(z_1)U(z_2) = U(z_1+z_2)$ for all $z_1,z_2\in S_\theta$.\\
(iii) & $U: z \mapsto U(z)$ is analytic function of $z\in S_\theta$. \\
(iv) & If  $x\in{\mathfrak{X}}$ and $0<\epsilon<\theta$, then $\lim_{z\rightarrow 0}U(z)\, x = x$
provided $z \in S_{\theta-\epsilon}$.
\end{tabular}
}}
\ed

Let $\sigma(A)={\dC}\setminus\rho(A)$ denote the spectrum of $A$. It can be used for the following
characterisation of the holomorphic semigroup generators, \cite{Kato95}, Chapter IX:
\bprop\label{prop:8.1.6}
Operator $A$ in a Banach space $\mathfrak{X}$ is the generator of a bounded holomorphic semigroup
of semi-angle $\theta \in (0, \pi/2]$ if and only if $A$ is a closed operator with a dense domain
${\rm{dom}}(A)$ such that:
\begin{equation*}
 \sigma(A) \subseteq \left\{ z \in{\dC},\ |\arg(z)|\leq
 {\pi\over 2}-\theta\right\} \ , \quad  0<\theta\leq{\frac{\pi}{2}} \ ,
\end{equation*}
and
\begin{equation*}\label{esres}
\|(z\,\mathds{1} +A)^{-1}\|\leq \frac{N_\epsilon }{|z|} \ \quad {\rm{for}}
\ \quad N_\epsilon > 0 \ , \
\  z \in S_{\theta+\pi/2 -\epsilon} \ \ .
\end{equation*}
where $0< \epsilon < \theta$.
\eprop

For applications it is also useful the following property, which is an alternative characterisation
of these kind of semigroups.
\bprop\label{prop:8.1.7}
If $U(z)$ is a bounded holomorphic semigroup of semi-angle $\theta$ with generator $A$, then for
all $z\in S_{\theta}$
and $n\in{\dN}$ one has $U(z){\mathfrak{X}}\subseteq{\rm{dom}}(A^n)$. Moreover, there are positive
constants $C_A'$, $C_A^{(n)}$ such that for $t>0$ \emph{:}
\begin{equation}\label{eq:8.1.11}
\left\|\frac{dU(t)}{dt}\right\| =  \|AU(t)\| \leq \frac{C_A'}{t} \ \ \  {\rm{and}} \ \ \
\left\|\frac{d^n U(t)}{dt^n}\right\| =
\left\|A^n U(t)\right\| \leq \frac{C_A^{(n)}}{t^n} \ .
\end{equation}
Let $0<\theta<\pi/2$. Then estimates \emph{(\ref{eq:8.1.11})} are valid for complex argument
$z \in S_\theta$ with constants depending on $\theta$.
\eprop
\br\label{rem:8.1.8}
{\rm{Similarly to strongly continuous semigroups, a family $U(z)$, $z\in S_\theta$ is called a
\textit{quasi-bounded} holomorphic semigroup of semi-angle $\theta$ if there exists a constant
$\beta >0$ such that {\blue{restriction of $\{\e^{-\beta \, z}\, U(z)\}_{z\in S_\theta}$ to
${\dR}^{+}_0$ is a bounded $C_0$-semigroup.}}
The class of semigroups that we consider here is restricted to \textit{holomorphic}
\textit{contraction} semigroups.}}
\er

To this aim we recall definition of this notion below, \cite{Kato95}, Chapter IX.
\bd \label{def:8.1.9}
{\rm{We say that $\{U_A(z)\}_{z\in S_\theta}$, is a \textit{holomorphic contraction} semigroup
with generator $A \in \mathscr{H}_{c}(\theta,0)$, {\blue{if its restriction
$\{U_A(t)\}_{t \geq 0}$ to ${\dR}^{+}_0$ is a contraction $C_0$-semigroup, that is
$\mathscr{H}_{c}(\theta, 0):= \mathscr{H}(\theta, 0) \cap \mathscr{G}(1,0)$.}}
}}
\ed

Note that this class of semigroups is \textit{not} empty and corresponding generators have the
following properties:\\
\smallskip
(i) Let $\{U_A(t)\}_{t \geq 0}$, be a \textit{contraction} semigroup with generator
$A\in \mathscr{G}(1,0)$ in a Banach space $\mathfrak{X}$, such that
$U_A(t){\mathfrak{X}}\subseteq{\rm{dom}}(A)$ for $t>0$. If $\|AU(t)\|\leq M_1 \, t^{-1}$ for some
$M_1 > 0$ and all $t>0$, then there exists $\theta = \arcsin\,(e M_1)^{-1} {(< \pi/2)}$
such that $U_A(t)$ may be analytically continued to contraction holomorphic semigroup
of semi-angle $\theta$.\\
(ii) Let $A$ be a sectorial operator in a Hilbert space $\mathfrak{H}$, i.e. its numerical range
$W= \{(Au,u): u\in{\rm{dom}}(A) \mbox{ and } \|u\|=1\}\subset S_{\pi/2-\theta}$ for
$0<\theta\leq \pi/2$.
If the operator $A$ is closed, then it is generator of the holomorphic contraction semigroup of
semi-angle $\theta$.
\\
(iii) Let $A$ be a generator of holomorphic semigroup on a Banach space $\mathfrak{X}$. If $A$ is
accretive, then $A$ generates a holomorphic contraction semigroup.
\\
(iv) {\blue{If $A$ is the generator of a strongly continuous \textit{group} $\{U_A(t)\}_{t \in \dR}$
of contractions $\|U_A(t)\| \leq 1$, then $\pm A \in \mathscr{G}(1,0)$, and
$A^2 \in \mathscr{H}_{c}(\pi/2,0)$, \cite{EN00}, Corollary II.4.9.
}}

\subsection{\blue{Fractional powers of generators}}\label{subsec:8.1.3}
We scrutinise in this section some properties of fractional powers of the generators for
bounded semigroups in a Banach space, see, e.g., \cite{Yos80}, Chapter IX.

Recall that fractional power $A^{\alpha}$, $0<\alpha<1$, of generator of a bounded $C_0$-semigroup
$U(t)$ ($\|U(t)\| \leq C_A$) can be expressed by the integral ({\blue{when it is well-defined}}):
\begin{equation}\label{eq:8.1.12}
A^{\alpha}x = {1\over \Gamma(-\alpha)}\int_0^\infty  d\lambda \, \lambda^{-\alpha-1} (U(\lambda) -
\mathds{1}) \ x  \ , \quad x\in \dom(A) ,
\end{equation}
where $\Gamma(\cdot)$ is the Gamma-function and $\lambda^{\alpha}$ is chosen to be positive for
$\lambda>0$. Since for any $x\in{\rm{dom}}(A)$ and $0 < \alpha \leq 1$ the integral (\ref{eq:8.1.12})
is convergent, ${\rm{dom}}(A)\subseteq{\rm{dom}}(A^\alpha)$. We set $A^0:=\mathds{1}$ and define
$A^\alpha=A^{\alpha-[\alpha]}A^{[\alpha]}$ for any $\alpha>0 \ $, where
$[\alpha]$ denotes the \textit{integer} part of $\alpha$,
\bprop\label{prop:8.1.10}
For each $\alpha\in [0,1]$, there exists a constant $C_{A,\alpha}$, depending only on $C_A$ and
$\alpha$, such that for all $\mu>0$,
\begin{equation}\label{eq:8.1.13}
\left\|A^{\alpha}(A+\mu \mathds{1})^{-1}\right\| \leq {C_{A,\alpha} \over\mu^{1-\alpha}}.
\end{equation}
\eprop
\begin{proof}
For $\alpha=0$ or $\alpha=1$, the result follows directly from the estimate of the resolvent.
Let $0<\alpha<1$ and
$x\in{\mathfrak{X}}$. Note that $\mbox{ran}(A+\mu \mathds{1})^{-1} =
{\rm{dom}}(A) \subseteq {\rm{dom}}(A^\alpha)$. Then
\begin{equation}\label{eq:8.1.14}
A^{\alpha}(A+\mu \mathds{1})^{-1}\ x =
{1\over \Gamma(-\alpha)}\int_0^\infty d\lambda \ \lambda^{-\alpha-1} (U(\lambda) - \mathds{1}))
(A + \mu \mathds{1})^{-1}\, x \ .
\end{equation}
We divide the integral (\ref{eq:8.1.14}) into two parts: $0<\lambda\leq\mu^{-1}$ and
$\lambda>\mu^{-1}$, and we use the
representation (\ref{eq:8.1.3}):
\begin{eqnarray*}
A^{\alpha}(A+\mu \mathds{1})^{-1}x & = & {1\over \Gamma(-\alpha)}\int_0^{\mu^{-1}} d\lambda \
\lambda^{-\alpha-1}
\int_0^\lambda dt \ (-U(t)) (\mathds{1}-\mu(A+\mu \mathds{1})^{-1})x \\
& & +\ {1\over \Gamma(-\alpha)}\int_{\mu^{-1}}^\infty d\lambda \ \lambda^{-\alpha-1}(U(\lambda) -
\mathds{1})
(A + \mu \mathds{1})^{-1}\ x \ .
\end{eqnarray*}
Now by the estimate of the resolvent $\|(A+\mu )^{-1}\|\leq C_A/\mu$ for all $\mu>0$ one obtains:
\begin{eqnarray*}
\|A^{\alpha}(A+\mu \mathds{1})^{-1}x\| & \leq & {C_A(1+C_A)\ \|x\|\over \Gamma(-\alpha)}
\left(\int_0^{\mu^{-1}} d\lambda \
\lambda^{-\alpha}  + {1\over\mu}\int_{\mu^{-1}}^\infty d\lambda \ \lambda^{-\alpha-1}  \right)\\
& \leq & \frac{C_A(1+C_A)\mu^{\alpha-1}}{\alpha(1-\alpha)\Gamma(-\alpha)} \ \|x\|.
\end{eqnarray*}
Setting $C_{A,\alpha} := C_A(1+C_A)/(\alpha(1-\alpha)\Gamma(-\alpha))$ we obtain the estimate
(\ref{eq:8.1.13}).
\hfill $\square$ \end{proof}

Next we recall the following well-known property of the semigroup generator $A$:
\bl\label{lem:8.1.11}
${\rm{dom}}((A+\delta \mathds{1})^\alpha) = {\rm{dom}}(A^\alpha)$ for all $\delta>0$ and $0<\alpha<1$.
\el
\bt\label{th:8.1.12}
Let $U_A(t)$ be a bounded holomorphic semigroup with generator $A$, then for any real $\alpha>0$,
we have
\begin{equation}\label{eq:8.1.15}
\sup_{t>0} \left\| t^\alpha A^\alpha U_A(t) \right\| = M_\alpha < \infty.
\end{equation}
\et
\begin{proof}
Let $0<\alpha<1$. By ${\rm{dom}}(A)\subseteq {\rm{dom}}(A^\alpha)$ one gets
${\rm{dom}}(A^\alpha U_A(t)) = {\mathfrak{X}}$.
Hence by (\ref{eq:8.1.12}) we have
\begin{equation}\label{eq:8.1.16}
A^\alpha U_A(t) = {1\over \Gamma(-\alpha)}\int_0^\infty d\lambda \ \lambda^{-\alpha-1}
(U_A(t+\lambda) - U_A(t)) \ .
\end{equation}
Now we split the integral (\ref{eq:8.1.16}) in two parts: $0< \lambda \leq t$ and $\lambda>t$,
and we use the estimate of the derivative of the holomorphic semigroup
(see Proposition \ref{prop:8.1.7}) to obtain
\begin{equation}
\|U_A(t+\lambda) - U_A(t)\| \leq \lambda\sup_{t\leq\tau\leq t+\lambda} \|\, \partial_{\tau}U_A(\tau)\|
\leq \lambda \ {C_A'\over t}.
\end{equation}
This leads to the estimate
\begin{eqnarray*}
\|A^\alpha U_A(t)\| & \leq & {1\over \Gamma(-\alpha)}\left(\int_0^t d\lambda \
\lambda^{-\alpha}\,{C_A'\over t} +
\int_t^\infty d\lambda \ 2\, C_A \, \lambda^{-\alpha-1}\right) \\
& \leq & {t^{-\alpha}\over \Gamma(-\alpha)}\left({C_A'\over 1-\alpha} + {2\, C_A\over\alpha}\right).
\end{eqnarray*}
Therefore one obtains (\ref{eq:8.1.15}) for $0<\alpha<1$ by setting $M_\alpha :=
\Gamma(-\alpha)^{-1}({C_A'/(1-\alpha)}$ $+$ ${2\, C_A/\alpha})$.

For integer powers $\alpha$, (\ref{eq:8.1.15}) follows directly from Proposition \ref{prop:8.1.7}.
Notice that by Proposition \ref{prop:8.1.7} $\mbox{ran}(U_A(t))\subseteq {\rm{dom}}(A^n)$ for $t>0$.
Then result
(\ref{eq:8.1.15}) follows for any non integer $\alpha>1$, from:
the observation that ${\rm{dom}}(A^\alpha = A^{\alpha-[\alpha]}A^{[\alpha]}) \supset
{\rm{dom}}(A^{[\alpha]+1})$, the representation (\ref{eq:8.1.16}), and the estimate (\ref{eq:8.1.11})
of derivatives of order $[\alpha]+1$.
\hfill $\square$ \end{proof}

\section{\blue{Operator-norm Trotter product formula}}\label{sec:8.2}
\subsection{\blue{Perturbation of holomorphic contraction semigroups}} \label{subsec:8.2.1}
In this section we prove the operator-norm convergence of the Trotter product formula in Banach space
$\mathfrak{X}$. Note that convergence of the abstract version of this formula in the strong
operator topology is known since \cite{Trot59}.
The proof goes via estimate of the rate of convergence in the case of \textit{small}
perturbations of the holomorphic contraction semigroups, cf. \cite{CZ01}.

We require that operator $A$ generates a holomorphic contraction semigroup, and that
perturbation $B$ satisfies the following hypothesis:

\smallskip
\noindent
(H1) $B$ is generator of a contraction semigroup on $\mathfrak{X}$.

\smallskip
\noindent
(H2) There is a real $\alpha\in[\, 0,1)$ such that ${\rm{dom}}(A^{\alpha})\subseteq{\rm{dom}}(B)$
and that
${\rm{dom}}(A^*)\subseteq{\rm{dom}}(B^*)$ adjoint operator in the dual space $\mathfrak{X}^*$.

\smallskip

Notice that we can suppose the operator $A$ boundedly invertible. If it is not the case, one considers
$A+\eta$ for some $\eta>0$. Then
by Lemma \ref{lem:8.1.11}  we have ${\rm{dom}}((A+\eta)^\alpha) =
{\rm{dom}}(A^\alpha) \subseteq {\rm{dom}}(B)$ .

\br \label{rem:8.2.1}
\rm{We note that assumption (H2) implies that $B$ is relatively bounded with respect to $A$
with the relative bound equals to \textit{zero}. Indeed, for $\eta>0$ by
${\rm{dom}}(A+\eta \mathds{1})\subseteq{\rm{dom}}(A^{\alpha})\subseteq{\rm{dom}}(B)$ and by
Proposition \ref{prop:8.1.10}, (\ref{eq:8.1.13}), one gets (here operator $A$ is supposed to be
boundedly invertible):
\begin{equation}\label{eq:8.2.1}
\|B\,(A+\eta \mathds{1})^{-1}\| \leq \|BA^{-\alpha}\| \, \|A^{\alpha}(A+\eta \mathds{1})^{-1}\| \leq
{C_\alpha\over\eta^{1-\alpha}}\, \|BA^{-\alpha}\|.
\end{equation}
Since the operators $A^{\alpha}$ and $B$ are closed, the inclusions in (H2) are equivalent to
$A^{\alpha}$-boundedness of
$B$ and the $A^*$-boundedness of $B^*$. In particular, $\|BA^{-\alpha}\|\leq d$ and
$\|B^*{A^{*}}^{-1}\|\leq d'$ for some $d,d'>0$.
Therefore, for any $x\in{\rm{dom}}(A)\subseteq{\rm{dom}}(B)$, we have the estimate
\begin{equation}\label{eq:8.2.2}
\|Bx\|\leq {C_\alpha \ d \over\eta^{1-\alpha}} \, \|Ax\| + \eta^\alpha C_\alpha \, d \ \|x\|
\end{equation}
and the relative bound in (\ref{eq:8.2.2}) can be made infinitesimally small for the large
enough shift parameter $\eta>0$.
}
\er

For this class of (small) perturbations of holomorphic contraction semigroup with generator
$A \in \mathscr{H}_{c}(\theta, 0)$ (Definition \ref{def:8.1.9}) one gets the following result:
\bl\label{lem:8.2.2}
Let $\{\e^{-zA}\}_{z \, \in \, S_\theta}$ be a holomorphic contraction semigroup of semi-angle
$\theta$ on $\mathfrak{X}$ and perturbation $B$ satisfy the hypothesis \emph{(H1)} and \emph{(H2)}.
Then the algebraic sum $A+B$ of operators defined on ${\rm{dom}}(A+B) = {\rm{dom}}(A)$
is also a generator of holomorphic contraction semigroup with the same semi-angle, that is,
$(A + B) \in \mathscr{H}_{c}(\theta, 0)$.
\el
\begin{proof}
To this end we verify conditions of Proposition \ref{prop:8.1.6}. Let $\epsilon \in (0,\theta)$.
Then by (\ref{eq:8.2.2}) we obtain inequality
\begin{equation*}
\|B\,(A+z \mathds{1})^{-1}\| \leq {C_\alpha \, \|B\,A^{-\alpha}\|\over \eta^{1-\alpha}}
\|A\,(A+z \mathds{1})^{-1}\| +
\eta^\alpha C_\alpha \, \|B\,A^{-\alpha}\| \, \|(A+z \mathds{1})^{-1}\| \ ,
\end{equation*}
for $|\arg(z)|<\theta + \pi/2 -\epsilon$. Seeing that
$\|(z\,\mathds{1} +A)^{-1}\|\leq N_\epsilon \, |z|^{-1}$ {for}
$N_\epsilon > 0$ and $z \in S_{\theta + \pi/2 -\epsilon} \, $, this inequality leads to
\begin{equation}\label{eq:8.2.3}
\|B\,(A+z \mathds{1})^{-1}\| \leq {C_\alpha\|B\,A^{-\alpha}\|\over \eta^{1-\alpha}} \,
(1+N_\epsilon) +
\eta^\alpha C_\alpha \, \|B\,A^{-\alpha}\| \ {N_\epsilon\over |z|}\, ,
\end{equation}
for $z \in S_{\theta + \pi/2 -\epsilon}\, $.
Therefore, the Neumann series for $(A+B+z \mathds{1})^{-1}$ converges if the right hand side of
(\ref{eq:8.2.3}) is smaller than $1$. Since we can choose $\eta$ and
$z \in S_{\theta + \pi/2 -\epsilon}$ such that the right-hand side of the
estimate (\ref{eq:8.2.3}) becomes smaller than $1$, we obtain:
\begin{equation*}
\|(A+B+z \mathds{1})^{-1}\| \leq {M\over |z-\gamma|} \, .
\end{equation*}
Here $M$ and $\gamma$ are some positive constants.
Then by Proposition \ref{prop:8.1.6} we conclude that operator
$(A+B) \in \mathscr{H}(\theta, \gamma)$,
that is, it generates a \textit{quasi-bounded} holomorphic semigroup
$\{U_{A+B}(z)\}_{z\in S_{\theta}}$.

On the other hand, the conditions of lemma imply that $A$ and $B$ are \textit{accretive}, thus
operator $A+B$ is also accretive. Since for $\lambda <0$, and $|\lambda|$ sufficiently large
($|\lambda|>\gamma$), the point $\lambda$ is in the resolvent set $\rho(A+B)$, we conclude
that $(A+B)\in \mathscr{G}(1,0)$ generates a contraction semigroup.

Since by $(A+B) \in \mathscr{H}(\theta, \gamma)$ this semigroup is also holomorphic, one
finally obtains the assertion.
\hfill $\square$
\end{proof}

The proof of the main theorem of this section involves three technical lemmata. For the two of
them we need \textit{only} that $B$ ($B^*$) are $A$($A^*$)-bounded in the Kato sense, i.e., there
are positive constants $a$ and $b$ such that:
\begin{eqnarray}
& &  x\in{\rm{dom}}(A)\subseteq{\rm{dom}}(B),\ \|Bx\|\leq a\|Ax\| + b\|x\|, \label{eq:8.2.4}\\
& &  \phi\in{\rm{dom}}(A^*)\subseteq{\rm{dom}}(B^*),\ \|B^*\phi\|\leq a\|A^*\phi\| + b\|\phi\|.
\label{eq:8.2.5}
\end{eqnarray}
If $A$ is boundedly invertible, then we can put $b=0$ with the relative bound $a + b\|A^{-1}\|$
instead of $a$.
\bl\label{lem:8.2.3}
Let boundedly invertible $A$ and operator $B$ be generators of bounded semigroups. Let $B$ and $B^*$
verify \emph{(\ref{eq:8.2.4}), (\ref{eq:8.2.5})} and suppose that operator $H=(A+B)$  with
${\rm{dom}}(H)={\rm{dom}}(A)$ is
the boundedly
invertible generator of a bounded semigroup. Then there exists constant $L_1$ such that for all
$\tau\geq 0$ :
\begin{eqnarray}
\left\|A^{-1}\left(\e^{-\tau B}\e^{-\tau A} - \e^{-\tau(A+B)}\right)\right\| & \leq &  L_1\tau,
\label{eq:8.2.6}\\
\left\|\left(\e^{-\tau B}\e^{-\tau A} - \e^{-\tau(A+B)}\right)A^{-1}\right\| & \leq &  L_1\tau.
\label{eq:8.2.7}
\end{eqnarray}
\el
\begin{proof}
By virtue of the identity
\begin{eqnarray*}
& & A^{-1}\left(\e^{-\tau B}\e^{-\tau A} - \e^{-\tau(A+B)}\right) = A^{-1}\left(\e^{-\tau B} -
\mathds{1}\right)\e^{-\tau A} \\
& & \hspace{4.0cm} + \  A^{-1}\left(\e^{-\tau A}-\mathds{1}\right) + A^{-1}HH^{-1}\left(\mathds{1} -
\e^{-\tau H}\right)
\end{eqnarray*}
and by Lemma \ref{lem:8.1.2} we get (\ref{eq:8.2.6}):
\begin{eqnarray*}
\left\|A^{-1}\left(\e^{-\tau B}\e^{-\tau A}-\e^{-\tau H}\right)\right\|
& \leq & \left\|\int_0^\tau ds \, A^{-1}B \, \e^{-sB}\right\| \\
& & \hspace{-2 cm} + \left\|A^{-1}\left(\e^{-\tau A}-\mathds{1}\right)\right\| +
\|A^{-1}H\| \, \left\|H^{-1} \left(\mathds{1} -
\e^{-\tau H}\right)\right\| \\
& \leq & \|A^{-1}B\|C_B\tau + C_A\tau + \|A^{-1}H\|C_H\tau \, ,
\end{eqnarray*}
where coefficients $C_B$ and $C_H$ are defined similarly  to $C_A$ in Subsection \ref{subsec:8.1.1}.

Finally we remark that (\ref{eq:8.2.5}) implies the boundedness of the closed operator $A^{-1}B$,
and that $\|A^{-1}H\|\leq \|\mathds{1}+ A^{-1}B\|\leq 1+a + b\|A^{-1}\|$. To prove (\ref{eq:8.2.7})
one has to use (\ref{eq:8.2.4}), and the same line of reasoning as above to put finally
$L_1 =C_B a' + C_A + C_H (1+a')$ where $a'=a+b\|A^{-1}\|$.
\hfill $\square$
\end{proof}
\bl\label{lem:8.2.4}
Let $A$, $B$ and $H=A+B$ be the same as in Lemma \emph{\ref{lem:8.2.3}}. Then there exists a
constant $L_2$ such that for all $\tau\geq 0$ :
\begin{eqnarray}
\left\|A^{-1}\left(\e^{-\tau B}\e^{-\tau A} - \e^{-\tau(A+B)}\right)A^{-1}\right\| & \leq & L_2\tau^2,
\label{eq:8.2.8}\\
\left\|A^{-1}\left(\e^{-\tau A}\e^{-\tau B} - \e^{-\tau(A+B)}\right)A^{-1}\right\| & \leq & L_2\tau^2.
\label{eq:8.2.9}
\end{eqnarray}
\el
\begin{proof}
By virtue of
\begin{eqnarray*}
& & \e^{-\tau B}\e^{-\tau A}-\e^{-\tau H} = \left(\mathds{1}-\e^{-\tau B}\right)
\left(\mathds{1}-\e^{-\tau A}\right) +
\left(\e^{-\tau A} - (\mathds{1} + \tau A)^{-1}\right) \\
& & \hspace{1.5cm}  + \left(\e^{-\tau B} - (\mathds{1} + \tau B)^{-1}\right) + \left((\mathds{1} +
\tau H)^{-1} -
\e^{-\tau H}\right) \\
& & \hspace{1.5cm} + \tau H(\mathds{1} + \tau H)^{-1} - \tau A(\mathds{1} + \tau A)^{-1} -
\tau B(\mathds{1} +
\tau B)^{-1}
\end{eqnarray*}
and by identity
\begin{eqnarray*}
& & A^{-1}\left(\tau H(\mathds{1} + \tau H)^{-1} - \tau A(\mathds{1} + \tau A)^{-1} -
\tau B(\mathds{1} +
\tau B)^{-1}\right)A^{-1} \\
& = & \tau^2\left((\mathds{1} +\tau A)^{-1} + A^{-1}B(\mathds{1} +\tau B)^{-1}BA^{-1} -
A^{-1}H(\mathds{1} +
\tau H)^{-1}HA^{-1}\right),
\end{eqnarray*}
we obtain the representation
\begin{eqnarray*}
& & A^{-1}(\e^{-\tau B}\e^{-\tau A}-\e^{-\tau H})A^{-1} = \\
& & \hspace{1.0cm}A^{-1}(\mathds{1} - \e^{-\tau B})(\mathds{1} - \e^{-\tau A})A^{-1} \\
& & \hspace{1.0cm} + \left(\e^{-\tau A} - (\mathds{1} + \tau A)^{-1}\right)A^{-2} +
A^{-1}\left(\e^{-\tau B} -
(\mathds{1} +\tau B)^{-1}\right)A^{-1} \\
& & \hspace{1.0cm}+A^{-1}H\left((\mathds{1} + \tau H)^{-1} - \e^{-\tau H}\right)H^{-2}HA^{-1} +
\tau^2(\mathds{1} +
\tau A)^{-1} \\
& & \hspace{1.0cm}+ \tau^2A^{-1}B(\mathds{1} + \tau B)^{-1}BA^{-1} - \tau^2A^{-1}H(\mathds{1} +
\tau H)^{-1}HA^{-1} \ .
\end{eqnarray*}
This presentation yields the following estimate:
\begin{eqnarray*}
& & {1\over\tau^2}\,\left\|A^{-1}(\e^{-\tau B}\e^{-\tau A}-\e^{-\tau H})A^{-1}\right\|\ \leq \\
& & \hspace{2cm}{1\over\tau}\,\|A^{-1}B\|\left\|\int_0^\tau ds \ \e^{-sB}\right\| \,
{1\over\tau}\, \left\| (\mathds{1} - \e^{-\tau A})A^{-1}\right\| \\
& & \hspace{2cm} + {1\over\tau^2}\, \left\|\left(\e^{-\tau A} - (\mathds{1} +
\tau A)^{-1}\right)A^{-2}\right\|  \\
& & \hspace{2cm} + {1\over\tau^2}\, \|A^{-1}B\|\left\|\int_0^\tau ds_1 \int_0^{s_1} ds_2 \
\e^{-s_2 B} - \tau^2 \, (\mathds{1}+\tau B)^{-1}\right\| \|BA^{-1}\| \\
& & \hspace{2cm} + {1\over\tau^2} \, \|A^{-1}H\|\left\|\left((\mathds{1} + \tau H)^{-1} -
\e^{-\tau H}\right)H^{-2}\right\| \|HA^{-1}\| \\
& & \hspace{2cm} + 1 + \|A^{-1}B\|\, \|BA^{-1}\| + \|A^{-1}H\|\, \|HA^{-1}\|.
\end{eqnarray*}
Now Lemmata \ref{lem:8.1.2} and \ref{lem:8.1.4} together with (\ref{eq:8.2.4}), (\ref{eq:8.2.5}) imply
(\ref{eq:8.2.8}),
where we can take $L_2 = a'C_A C_B + {3C_A/ 2} + 3C_B a'^2/2 + 3C_H (1+a')^2/2 + 1 + a'^2 +
(1+a')^2$ with
$a'=a+b\|A^{-1}\|$.
Similarly one obtains (\ref{eq:8.2.9}).
\hfill $\square$
\end{proof}

Note that for the proof of the third lemma (Lemma \ref{lem:8.2.4}) we do need \textit{conditions}
(H2), as well as requirement that semigroups are \textit{contractive}.

\bl\label{lem:8.2.5}
Let $A$ be a boundedly invertible generator of holomorphic contraction semigroup. If $B$ is
generator of a
contraction semigroup and there exists $\alpha\in [0,1)$ such that
${\rm{dom}}(A^{\alpha})\subseteq{\rm{dom}}(B)$,
then
for any $k\geq 1$ and $\tau>0$ one gets the estimates
\begin{eqnarray}
\left\|\left(\e^{-\tau B}\e^{-\tau A}\right)^k A\right\| & \leq & {L_3\over\tau^{\alpha}} +
{C_A'\over k\tau}\ ,
\ \ \ \ \ \
 \ \ \ \ \ \ \ \ \ {\rm{if}} \ \ \alpha>0 \ , \label{eq:8.2.10}\\
\left\|\left(\e^{-\tau B}\e^{-\tau A}\right)^k A\right\|\ & \leq & {\tilde{L}_3(1+\ln k)} +
{C_A'\over k\tau}\ , \ \
{\rm{if}} \ \ \alpha=0 \ . \label{eq:8.2.11}
\end{eqnarray}
\el
\begin{proof}
We start with the following chain of estimates:
\begin{eqnarray*}
& & \left\|\left(\e^{-\tau B}\e^{-\tau A}\right)^k A\right\|\leq \left\|\left(\left(\e^{-\tau B}
\e^{-\tau A}\right)^k -
\e^{-k\tau A}\right)A\right\| + \left\|\e^{-k\tau A}A\right\| \\
& & \hspace{1.0cm} \leq \left\|\sum_{j=0}^{k-1}\left(\e^{-\tau B}\e^{-\tau A}\right)^{k-1-j}
\left(\e^{-\tau B}-I\right)
\e^{-\tau A} \e^{-j\tau A}A\right\| + \left\|\e^{-k\tau A}A\right\| \\
& & \hspace{1.0cm}\leq \sum_{j=0}^{k-1}\left\|\int_0^\tau ds \ \e^{-sB}B\, A^{-\alpha}\right\|
\left\|A^{\alpha}\e^{-(j+1) \tau A}A\right\| + \left\|\e^{-k\tau A}A\right\| \ .
\end{eqnarray*}
Notice that the second inequality is in particular due to contractions of $\e^{-tA}$ and $\e^{-tB}$,
and to equation
(\ref{eq:8.1.3}) of Lemma \ref{lem:8.1.2}. {From} the hypothesis ${\rm{dom}}(A^{\alpha})
\subseteq{\rm{dom}}(B)$ we
deduce that $\|BA^{-\alpha}\|\leq d$, see Remark \ref{rem:8.2.1}. By Propositions \ref{prop:8.1.7}
and Theorem
\ref{th:8.1.12} we get respectively:
$$
\left\|\e^{-k\,\tau A}A\right\|\leq{C_A' \over k\,\tau} \ \ \ \mbox{ and } \ \
\left\|A^{1+\alpha}\e^{-(j+1)\,\tau A}\right\|\leq {M_\alpha
\over ((j+1)\tau)^{1+\alpha}} \ .
$$
Therefore, we conclude that:
\begin{eqnarray*}
& & \left\|\left(\e^{-\tau B}\e^{-\tau A}\right)^k A\right\|\leq
{M_\alpha d\over\tau^{\alpha}} \sum_{j=0}^{k-1}{1\over (j+1)^{1+\alpha}} + {C_A'\over k\,\tau} \, .
\end{eqnarray*}
Since $\alpha>0$, this gives the announced result (\ref{eq:8.2.10}) with

$$L_3 = d \, M_\alpha  \sum_{j=1}^{\infty}(1/j)^{1+\alpha}\ ,$$

and (\ref{eq:8.2.11}) for $\alpha=0$ with $\tilde{L}_3 = \|B\| \, C_A'$.
\hfill $\square$ \end{proof}

Note that since ${\rm{dom}}(A^\alpha)\subseteq{\rm{dom}}(B)$ implies ${\rm{dom}}(A^{\alpha'})
\subseteq {\rm{dom}}(B)$ for
$\alpha'\geq\alpha$, the estimate (\ref{eq:8.2.10}) is valid in fact for any $\alpha'\geq \alpha$.

\subsection{\blue{Convergence rate}} \label{subsec:8.2.2}
\bt\label{th:8.2.6}
Let $\{\e^{-zA}\}_{z \, \in \, S_\theta}$ be a holomorphic contraction semigroup, that is,
$A \in \mathscr{H}_{c}(\theta, 0)$. Let $B$ be generator of a
contraction semigroup. If there exists $\alpha\in [0,1)$ such that
${\rm{dom}}(A^{\alpha})\subseteq{\rm{dom}}(B)$ and ${\rm{dom}}(A^*)\subseteq{\rm{dom}}(B^*)$,
then there are constants $M_1, M_2, \tilde{M}_2, \eta>0$, such that for
any $t\geq 0$ and $n>2$ one gets estimates
\begin{eqnarray}
\left\|\left(\e^{-tB/n}\e^{-tA/n}\right)^n - \e^{-t(A + B)}\right\| & \leq & \e^{\eta t}(M_1 +
M_2 \, t^{1-\alpha})\ {\ln n\over n^{1-\alpha}}\ , \ \alpha>0,\label{eq:8.2.12}\\
\left\|\left(\e^{-tB/n}\e^{-tA/n}\right)^n - \e^{-t(A + B)}\right\| &
\leq & \e^{\eta t}(M_1 + \tilde{M}_2 \, t)\
{2\, (\ln n)^2\over n}\ , \ \ \alpha=0.\label{eq:8.2.13}
\end{eqnarray}
\et
\begin{proof}
Since $B$ satisfies hypothesis (H1) and (H2), by Lemma \ref{lem:8.2.2} the operator $H = (A+B)$ is
generator of a holomorphic contraction
semigroup. If operator $A$ has no bounded inverse, let $\tilde{A}:=A + \eta$ and
$\tilde{H}:=\tilde{A} + B$ for some $\eta >0$
(see Remark \ref{rem:8.2.1}). Then both operators are boundedly invertible. As we indicated above,
these changes of
generators do not modify the domain inclusions. If we want to obtain $\|B\tilde{A}^{-1}\|<1$ then by
the estimate
(\ref{eq:8.2.1}) we have to choose a sufficiently large shift parameter $\eta > 0$. This gives us
the estimate
$\|\tilde{A}\tilde{H}^{-1}\| = \|(\mathds{1} + B\tilde{A}^{-1})^{-1}\| \leq {1/ (1-a)}$ where we set
$a=\|B\tilde{A}^{-1}\|$.

Now we put $\tau := t/n$, $\tilde{U}(t):=\e^{-t\tilde{H}}$, and $\tilde{T}(\tau) :=
\e^{-\tau B}\e^{-\tau\tilde{A}}$.
To estimate the left-hand side of (\ref{eq:8.2.12}) we use
\begin{equation}\label{eq:8.2.14}
\left(\e^{-tB/n}\e^{-tA/n}\right)^n - \e^{-t(A + B)} = (\tilde{T}^n(\tau) -
\tilde{U}^n(\tau)) \ \e^{t\eta} \ ,
\end{equation}
and \textit{telescopic} identity:
\begin{eqnarray*}
\tilde{T}(\tau)^n-\tilde{U}(\tau)^n & = & \sum_{m=0}^{n-1}\tilde{T}(\tau)^{n-m-1}(\tilde{T}
(\tau)-\tilde{U}(\tau))
\tilde{U}(\tau)^m \\
& = & \tilde{T}(\tau)^{n-1}\tilde{A}\tilde{A}^{-1}(\tilde{T}(\tau)-\tilde{U}(\tau)) \\
& & + (\tilde{T}(\tau)-\tilde{U}(\tau))\tilde{A}^{-1}\tilde{A}\tilde{H}^{-1}\tilde{H}
\tilde{U}(\tau)^{n-1}\\
& &  + \sum_{m=1}^{n-2}\tilde{T}(\tau)^{n-m-1}\tilde{A}\tilde{A}^{-1}(\tilde{T}(\tau)-
\tilde{U}(\tau))\tilde{A}^{-1}
\tilde{A}\tilde{H}^{-1}\tilde{H}\tilde{U}(\tau)^m \ ,
\end{eqnarray*}
which implies
\begin{eqnarray*}
& & \left\|\tilde{T}(\tau)^n-\tilde{U}(\tau)^n\right\|\leq
\|\tilde{T}(\tau)^{n-1}\tilde{A}\| \, \|\tilde{A}^{-1}(\tilde{T}(\tau)-\tilde{U}(\tau))\| \\
& & \hspace{1.5cm} + \|(\tilde{T}(\tau)-\tilde{U}(\tau))\tilde{A}^{-1}\| \,
\|\tilde{A}\tilde{H}^{-1}\| \, \|\tilde{H}\tilde{U}(\tau)^{n-1}\|\\
& & \hspace{1.5cm} + \sum_{m=1}^{n-2}\|\tilde{T}(\tau)^{n-m-1}\tilde{A}\| \, \|\tilde{A}^{-1}
(\tilde{T}(\tau)-\tilde{U}
(\tau))\tilde{A}^{-1}\| \, \|\tilde{A}\tilde{H}^{-1}\| \, \|\tilde{H}\tilde{U}(\tau)^m\| \, .
\end{eqnarray*}
Hence by Lemmata \ref{lem:8.2.3}, \ref{lem:8.2.4}, and \ref{lem:8.2.5} (it is at this point that
we use the
hypothesis of \textit{contraction}), and by Proposition \ref{prop:8.1.7} we obtain the estimate :
\begin{eqnarray}
& & \|\tilde{T}(\tau)^n-\tilde{U}(\tau)^n\| \leq \left({L_3\over\tau^\alpha} +
{C_A'\over(n-1)\tau}\right)L_1\tau +
{L_1\over 1-a}{C_H'\over n-1}\nonumber\\
& & \hspace{3.5cm} + \sum_{m=1}^{n-2}\left(L_3 \tau^{1-\alpha} +
{C_A'\over n-m-1}\right){L_2\over 1-a}{C_H'\over m}
\nonumber\\
& & \hspace{2.0cm}\leq L_3 L_1 {t^{1-\alpha}\over n^{1-\alpha}} +
{L_1\over n-1}\left(C_A'+{C_H'\over 1-a}\right) +
{L_3 L_2 C_H'\over 1-a}{t^{1-\alpha}\over n^{1-\alpha}} \sum_{m=1}^{n-2}{1\over m}\nonumber\\
& & \hspace{3.0cm}+ {L_2 C_H' C_A'\over 1-a}\sum_{m=1}^{n-2}{1\over n-m-1}\cdot{1\over m}\nonumber\\
& & \hspace{2.0cm}\leq L_3 L_1 {t^{1-\alpha}\over n^{1-\alpha}} +
{L_1\over n-1}\left(C_A'+{C_H'\over 1-a}\right)
\nonumber\\
& & \hspace{3.0cm} + 2 \, {L_3 L_2 C_H'\over 1-a}\, t^{1-\alpha}\ {\ln n\over n^{1-\alpha}} +
4{L_2 C_H' C_A'\over 1-a}\ {\ln n\over n}. \label{eq:8.2.15}
\end{eqnarray}
Here we used that:
\begin{equation*}
\sum_{m=1}^{n-1} {1\over (n-m)m} = {2\over n}\sum_{m=1}^{n-1} {1\over m} \leq {2\over n}(1 +
\ln(n-1)) \leq 4{\ln n\over n}.
\end{equation*}
The estimate (\ref{eq:8.2.15}) and (\ref{eq:8.2.14}) imply the announced result (\ref{eq:8.2.12})
for $\alpha>0$, with
$ M_1=4\, L_1\left(C_A'+ {\displaystyle C_H'\over \displaystyle 1-a}\right) +
4\, {\displaystyle L_2 C_H' C_A'\over\displaystyle 1-a}$ and $M_2 =
2\, L_3 L_1+2\, {\displaystyle L_3 L_2 C_H'\over\displaystyle 1-a}$.

In a similar way one gets also estimate for $\alpha=0$:
\begin{eqnarray*}
& & \|\tilde{T}(\tau)^n-\tilde{U}(\tau)^n\| \leq  \tilde{L}_3(1+\ln(n-1))\ L_1\, {t\over n} +
{L_1 C_A'\over n-1} +
{L_1{C}_H'\over 1-a}{1\over n-1}\\
& & \hspace{2.5cm} + \sum_{m=1}^{n-2}\left(\tilde{L}_3{t\over n} (1+\ln(n-m-1)) +
{C_A'\over n-m-1}\right)
{L_2\over 1-a}{C_H'\over m}\, .
\end{eqnarray*}
This estimate together with (\ref{eq:8.2.14}) yield (\ref{eq:8.2.13}) for
$\tilde{M}_2=2 \, \tilde{L}_3 L_1+ 2 \, {\displaystyle\tilde{L}_3 L_2 C_H'\over\displaystyle 1-a}$.
\hfill $\square$
\end{proof}
\bt\label{th:8.2.7}
Let $\{\e^{-zA}\}_{z \, \in \, S_\theta}$ be a holomorphic contraction semigroup, that is,
$A \in \mathscr{H}_{c}(\theta, 0)$. Let $B$ be  generator of a contraction
semigroup, and there exists $\alpha\in [0,1)$ such that
${\rm{dom}}((A^{\alpha})^*)\subseteq{\rm{dom}}(B^*)$ and ${\rm{dom}}(A) \subseteq{\rm{dom}}(B)$.
If in addition ${\rm{dom}}(A^*)\subseteq{\rm{dom}}(B^*)$ (for the case, when the space
$\mathfrak{X}$ is not reflexive), then there are constants $M_3$, $M_4$, $\tilde{M}_4$, $\eta>0$,
such that for any $t\geq 0$ and $n>2$:
\begin{eqnarray}
\left\|\left(\e^{-tA/n}\e^{-tB/n}\right)^n - \e^{-t(A + B)}\right\| & \leq & \e^{\eta t}(M_3 +
M_4 \, t^{1-\alpha})\
{\ln n\over n^{1-\alpha}}, \ \alpha>0\, , \label{eq:8.2.16}\\
\left\|\left(\e^{-tA/n}\e^{-tB/n}\right)^n - \e^{-t(A + B)}\right\| & \leq & \e^{\eta t} (M_3 +
\tilde{M}_4 \, t)\
{2\, (\ln n)^2\over n}, \ \alpha=0\, , \label{eq:8.2.17}
\end{eqnarray}
for any $t\geq 0$ and $n>2$.
\et
\begin{proof}
Let $\tilde{F}(\tau):=\e^{-\tau \tilde{A}}\e^{-\tau B}$. Then by the same arguments as in the
proof of Theorem \ref{th:8.2.6}, one obtains:
\begin{eqnarray*}
\tilde{U}(\tau)^n-\tilde{F}(\tau)^n & = & \sum_{m=0}^{n-1}\tilde{U}(\tau)^{n-m-1}(\tilde{U}
(\tau)-\tilde{F}
(\tau))\tilde{F}(\tau)^m \\
& = & \tilde{U}(\tau)^{n-1}\tilde{H}\tilde{H}^{-1}\tilde{A}\, \tilde{A}^{-1}(\tilde{U}(\tau)-\tilde{F}
(\tau)) \\
& & + (\tilde{U}(\tau)-\tilde{F}(\tau))\tilde{A}^{-1}\tilde{A}\tilde{F}(\tau)^{n-1}\\
& &  + \sum_{m=1}^{n-2}\tilde{U}(\tau)^{n-m-1}\tilde{H}\tilde{H}^{-1}\tilde{A}\,\tilde{A}^{-1}
(\tilde{U}(\tau)-
\tilde{F}
(\tau))\tilde{A}^{-1}\tilde{A}\tilde{F}(\tau)^m.
\end{eqnarray*}
Notice that the Lemmata \ref{lem:8.2.3} and \ref{lem:8.2.4} hold for $\tilde{F}(\tau)$.
By a simple modification
of Lemma \ref{lem:8.2.5}, where one uses $\|\tilde{A}^{-\alpha}B\| =
\|B^*(\tilde{A}^{-\alpha})^*\|<\infty$,
we find that

\begin{eqnarray*}
\left\|\tilde{A}\left(\e^{-\tau \tilde{A}}\e^{-\tau B}\right)^k\right\| & \leq &
{L_4\over \tau^\alpha} +
{C_A'\over k\, \tau}\ , \ \ \ \ \ \ \ \ \ \ \ \ \ \ \ \ \alpha>0,\\
\left\|\tilde{A}\left(\e^{-\tau \tilde{A}}\e^{-\tau B}\right)^k\right\| & \leq &
{\tilde{L}_4(1+\ln k)} + {C_A'\over k\, \tau}\ , \ \ \ \  \alpha=0.
\end{eqnarray*}
These ingredients ensure the estimates (\ref{eq:8.2.16}) and (\ref{eq:8.2.17}).
\hfill $\square$
\end{proof}

\bc\label{cor:8.2.8}
Under the same conditions as in Theorem \emph{\ref{th:8.2.6}}, we have the operator-norm
convergence of the symmetrised Trotter formula, i.e., there exists $M_5$, $M_6$, $\tilde{M}_6$,
$\eta>0$, such that for any $t\geq 0$ and $n>2$:
\begin{eqnarray}
\left\|\left(\e^{-tA/2n}\e^{-tB/n}\e^{-tA/2n}\right)^n - \e^{-t(A + B)}\right\| & \leq &
\e^{\eta t}\, (M_5 +
M_6 \, t^{1-\alpha})\ {\ln n\over n^{1-\alpha}} \ ,  \nonumber \\
&& {\rm{for}} \ \ \  0 < \alpha < 1 \ , \ \ \ \label{eq:8.2.18}   \\
\left\|\left(\e^{-tA/2n}\e^{-tB/n}\e^{-tA/2n}\right)^n - \e^{-t(A + B)}\right\| & \leq &
\e^{\eta t}\, (M_5 + \tilde{M}_6 \,t)\ {2\,(\ln n)^2\over n} \ ,  \nonumber \\
&& {\rm{for}} \ \ \ \alpha=0\ . \ \ \ \label{eq:8.2.19}
\end{eqnarray}
\ec
\begin{proof}
Since Lemmata \ref{lem:8.2.3}, \ref{lem:8.2.4}, and \ref{lem:8.2.5} can be easily extended to the
symmetrized product
$\e^{-\tau A/2}\e^{-\tau B}\e^{-\tau A/2}$, the proof of the Theorem \ref{th:8.2.6} carries through
verbatim to obtain (\ref{eq:8.2.18}) and (\ref{eq:8.2.19}).
\hfill $\square$
\end{proof}
\br \label{rem:8.2.9}
\rm{
Seeing that in Theorems \ref{th:8.2.6}, \ref{th:8.2.7} and in Corollary \ref{cor:8.2.8} the
\textit{perturbation} $B$ of \textit{dominating} operator $A$ is either \textit{infinitesimally}
$A$-small,
or simply \textit{bounded}, the corresponding results in Banach space $\mathfrak{X}$ are
\textit{weaker} than those in Hilbert space $\mathfrak{H}$, see \cite{Rog93},
\cite{NZ98}, \cite{NZ99}, and \cite{Zag20}.
Recall that in \cite{Rog93}, \cite{NZ98} the
perturbation $B$ in $\mathfrak{H}$ is \textit{Kato-small} with respect to operator $A$
for relative bound $b < 1$. The \textit{fractional} condition (H2) in a Hilbert space $\mathfrak{H}$
was introduced in \cite{IT97}.
Note that in the both cases: $\mathfrak{X}$ and $\mathfrak{H}$, the
\textit{dominating} operator $A$ is supposed to be generator of a \textit{holomorphic} semigroup.}
\er
\section{\blue{Example}} \label{sec:8.3}
Resuming Remark \ref{rem:8.2.9}, the question arises: whether the Trotter product formula converges
in the \textit{operator-norm} topology if the condition on dominating generator
$A \in \mathscr{H}_{c}(\theta, 0)$ is relaxed to hypothesis that $A \in \mathscr{G}(1,0)$, i.e.,
it is generator of a \textit{contraction} (but \textit{not} holomorphic!)
semigroup and $B$ is a \textit{bounded} generator ?

The aim of this section is to give an answer to this question using \textit{example} of a certain
class of generators and semigroups. It turns out that appropriate for this purpose is the class
of generators of \textit{evolution semigroups}.
\subsection{\blue{Evolution semigroups and Trotter product formula}} \label{subsec:8.3.1}
To proceed further we need some key notions from the \textit{evolution semigroups} theory and
in particular the notion of \textit{solution operator}.

A strongly continuous mapping $U(\cdot,\cdot): \gD \longrightarrow \cL(\mathfrak{X})$,
where domain $\gD := \{(t,s): 0 < s \le t \le T\}$ and $\cL(\mathfrak{X})$ is the set of bounded
operators on \textit{separable} Banach space $\mathfrak{X}$, is called a
\textit{solution operator} if the conditions
%
\begin{eqnarray*}
&& {\rm{(i)}} \ \sup_{(t,s)\in\gD}\|U(t,s)\|_{\cL(\mathfrak{X})} < \infty \ , \\
&& {\rm{(ii)}} \ U(t,s) = U(t,r)U(r,s) \, , \  \ 0 < s \le r \le t \le T \, ,
\end{eqnarray*}
%
are satisfied. Let us consider the Banach space $L^p(\cI,\mathfrak{X})$ for $\cI := [0,T]$ and
$p \in [1,\infty)$. The operator $\cK$ is an \textit{evolution generator} of the evolution semigroup
$\{\cU(\gt):= \e^{-\gt \cK}\}_{\tau\geq0}$ if there is a solution operator such that the
\textit{Howland-Evans-Neidhardt} representation, see \cite{How74}, \cite{Ev76}, \cite{Nei79}
and \cite{Nei81}:
\begin{equation}\label{eq:1.5}
(\e^{-\gt \cK}f)(t) = U(t,t-\gt)\chi_{\cI}(t-\gt)f(t-\gt), \quad f \in L^p(\cI,\mathfrak{X}) \ ,
\end{equation}
holds for a.a. $t \in \cI$ and $\gt \ge 0$. Seeing that on account of (\ref{eq:1.5}) the
semigroup $\{\e^{-\gt \cK}\}_{\tau\geq0}$ is \textit{nilpotent}: $\e^{-\gt \cK} f = 0$ for
$\gt \ge T$, the evolution generator $\cK$ can never be generator of a holomorphic semigroup.

A simple example of an evolution generator is the differentiation operator, cf. \cite{NSZ20}:
\begin{equation}\label{eq:1.6}
\begin{split}
(D_0f)(t) &:= \partial_{t} f(t), \\
f \in \dom(D_0) &:= \{f \in W^{1,p}(\cI,\mathfrak{X}): f(0) = 0\}\, ,
\end{split}
\end{equation}
where $W^{1,p}(\cI,\mathfrak{X})$ is the \textit{Sobolev} space of order $(1,p)$ of \textit{Bochner}
$p$-integrable functions.
Then by (\ref{eq:1.6}) one obviously gets the contraction shift semigroup:
\begin{equation}\label{eq:1.61}
(\e^{-\gt D_0}f)(t) = \chi_{\cI}(t-\gt)f(t-\gt), \quad f \in L^p(\cI,\mathfrak{X}),
\end{equation}
for a.a. $t \in \cI$ and $\gt \ge 0$. Hence, (\ref{eq:1.5}) implies that the corresponding
solution operator of the \textit{non-holomorphic} evolution semigroup $\{\e^{-\gt D_0}\}_{\tau\geq0}$
is given by $U_{D_0}(t,s) = \1$, for all $(t,s) \in \gD$.

Below we consider the operator $\cK_0 := \overline{D_0 + \cA}\, $,
where $\cA$ is the \textit{multiplication} operator induced by generator $A$ of a
\textit{holomorphic contraction} semigroup on $\mathfrak{X}$. More precisely
\begin{equation*}
\begin{split}
(\cA f)(t) &:= Af(t), \ {\rm{and}} \ (\e^{-\gt \cA} f)(t) = \e^{-\gt A}  f(t) \ , \\
f \in \dom(\cA) &:= \{f \in L^p(\cI,\mathfrak{X}): Af \in L^p(\cI,\mathfrak{X})\} \ .
\end{split}
\end{equation*}
Then the perturbation of the shift semigroup (\ref{eq:1.61}) by $\cA$ corresponds to the semigroup
with generator $\cK_0$. One easily checks that $\cK_0$ is an evolution generator of a contraction
semigroup on $L^p(\cI,\mathfrak{X})$, that is never holomorphic \cite{NSZ20}. Indeed, since the
generators $D_0$ and $\cA$ commute, the representation (\ref{eq:1.5}) for evolution semigroup
$\{\e^{-\gt \cK_0}\}_{\tau\geq0}$ takes the form:
\begin{equation*}
(\e^{-\gt \cl K_0}f)(t) = \e^{-\gt A}\chi_{\cI}(t-\gt)f(t-\gt), \quad f \in L^p(\cI,\mathfrak{X}) \ ,
\end{equation*}
for a.a. $t \in \cI = [0,T]$ and $\gt \ge 0$. Then by (\ref{eq:1.5}) the
solution operator $U_{0}(t,s) = \e^{-(t-s) A}\, $.
Therefore, $\e^{-\gt \cK_0} f = 0$ for $\gt \ge T$, that is, semigroup
$\{\e^{-\gt \cK_0}\}_{\tau\geq0}$ is \textit{nilpotent}.

Furthermore, if $\{B(t)\}_{t \in \cI}$ is a \textit{strongly measurable} family of generators of
contraction semigroups on $\mathfrak{X}$, that is, $B(\cdot): \cI \longrightarrow \cG(1,0)$, then the
induced \textit{multiplication} operator $\cB$ :
\begin{align}\label{eq:1.65}
(\cB f)(t) &:= B(t) f(t) \ ,\\
f \in \dom(\cB) &:= \left\{f \in L^p(\cI,\mathfrak{X}):
\!\!\!\!\begin{matrix} \ \
& \ \ f(t) \in \dom(B(t)) \ \ \ \mbox{for a.a.} \ t \in \cI\\
& B(t)f(t) \in L^p(\cI,\mathfrak{X})
\end{matrix}\right\}\, ,\nonumber
\end{align}
is a generator of a contraction semigroup on $L^p(\cI,\mathfrak{X})$.

In the next Subsection \ref{subsec:8.3.2} we consider perturbation of generator $\cK_0$
by multiplication operator $\cB$ (\ref{eq:1.65}). Thereupon we construct by means of the
\textit{Trotter} product formula approach a corresponding perturbed semigroup.
\br \label{rem:8.3.1}
\rm{
We conclude by remarks concerning some \textit{notations} and \textit{definitions} that we use
below throughout Section \ref{sec:8.3}.
\begin{enumerate}
\item For characterisation the rate of convergence we use, so-called, \textit{Landau's} symbols:
\begin{align}
g(n) &= O(f(n)) \Longleftrightarrow \limsup_{n\to\infty} \left|\frac{g(n)}{f(n)}\right| < \infty \ ,
\nonumber \\
g(n) &= o(f(n)) \Longleftrightarrow \limsup_{n\to\infty} \left|\frac{g(n)}{f(n)}\right| = 0 \ ,
\nonumber \\
g(n) &= \gT(f(n)) \Longleftrightarrow 0 < \liminf_{n\to\infty} \left|\frac{g(n)}{f(n)}\right|
\le \limsup_{n\to\infty} \left|\frac{g(n)}{f(n)}\right| < \infty \ , \nonumber \\
g(n) &=  \go(f(n)) \Longleftrightarrow \limsup_{n\to\infty} \left|\frac{g(n)}{f(n)}\right| = \infty \ .
\nonumber
\end{align}

\item We also use the notation $C^{0,\beta}(\cl I):=\{f:\cl I\rightarrow \C \ :
\mathrm{there~exists~some~~ } K > 0 \mathrm{~~such~that~ } |f(x)-f(y)|\leq K \ |x-y|^\beta ,
\mathrm{~~for~any ~} x,y \in \cl I \mathrm{~and~} \beta \in (0,1] \, \}\, $.

\item
Below we consider the Banach space $L^p(\cI,\mathfrak{X})$ for  $\cI := [0, T]$ and
$p \in [1,\infty)$.
\end{enumerate}
}
\er
\subsection{\blue{Trotter product formula}} \label{subsec:8.3.2}
We reminisce (cf. Subsection \ref{subsec:8.3.1}) that semigroup $\{\cU(\gt)\}_{\gt \ge 0}$, on the
Banach space $L^p(\cI,\mathfrak{X})$ is called the \textit{evolution semigroup} if there is a
\textit{solution operator}: $\{U(t,s)\}_{(t,s)\in\gD}$, such that representation (\ref{eq:1.5})
holds.

Let $\cK_0$ be the generator of an evolution semigroup $\{\cU_0(\gt)\}_{\gt \ge 0}$ and let $\cB$
be a multiplication operator induced by a measurable family $\{B(t)\}_{t\in\cI}$ of generators of
contraction semigroups. Note that in this case the multiplication operator $\cB$ (\ref{eq:1.65}) is a
generator of a contraction semigroup $(\e^{- \tau \, \cB} f)(t) = \e^{- \tau \, B(t)} f(t)$,
on the Banach space $L^p(\cI,\mathfrak{X})$. Since $\{\cU_0(\gt)\}_{\gt \ge 0}$ is an evolution
semigroup, then by definition (\ref{eq:1.5}) there is a propagator $\{U_0(t,s)\}_{(t,s) \in \gD}$
such that the representation
\begin{equation*}
(\cU_0(\gt)f)(t) = U_0(t,t-\gt) \ \chi_\cI(t-\gt)f(t-\gt), \quad f \in L^p(\cI,\mathfrak{X}),
\end{equation*}
is valid for a.a. $t \in \cI$ and $\gt \ge 0$. Then we define $\tau_{n} := (t-s)/n$, for $n \in \dN$,
and
\begin{equation*}
G_j(t,s;n) := U_0(s + j \, \tau_{n}\, , \ s + (j-1)\, \tau_{n}) \
\e^{- \ \tau_{n} \, B \, \big(s \, + \, (j-1)\ \tau_{n}\big)} \, , \quad (t,s) \in \gD \, ,
\end{equation*}
where $j \in \{1,2,\ldots,n\}$, $n \in \dN$, $(t,s) \in \gD$, and we set
\begin{equation*}
V_n(t,s) := \prod^{1}_{j=n}G_j(t,s;n), \quad n \in \dN, \quad (t,s) \in \gD \, .
\end{equation*}
That is, the product is increasingly ordered in $j$ from the right to the left.
Then a straightforward computation shows that the representation
\begin{equation}\label{eq:2.01}
\left(\left(\e^{-\gt \cK_0/n}\e^{-\gt \cB/n}\right)^n f\right)(t) =
V_n(t,t-\gt) \ \chi_\cI(t-\gt)f(t-\gt) \ ,
\end{equation}
$f \in L^p(\cI,\mathfrak{X})$, holds for each $\gt \ge 0$ and a.a. $t \in \cI$.
\begin{theorem}\label{th:2.1}
Let $\cK$ and $\cK_0$ be generators of evolution semigroups on the Banach space
$L^p(\cI,\mathfrak{X})$ for some
$p \in [1,\infty)$. Further, let $\{B(t)\in \cG(1,0)\}_{t\in \cI}$ be a strongly measurable family of
generators of contraction semigroups on $\mathfrak{X}$. Then
\bea\label{eq:2.0}
\lefteqn{
\sup_{\gt\ge 0}\left\|\e^{-\gt \cK} -
\left(\e^{-\gt \cK_0/n}\e^{-\gt \cB/n}\right)^n\right\|_{\cL(L^p(\cI,\mathfrak{X}))}}\\
& &  = \esssup_{(t,s)\in \gD}\|U(t,s) - V_n(t,s)\|_{\cL(\mathfrak{X})}, \quad n\in \dN.\nonumber
\eea
\end{theorem}
\begin{proof}
Let $\{L(\gt)\}_{\gt \ge 0}$ be the left-shift semigroup on the Banach space $L^p(\cI,\mathfrak{X})$:
\begin{equation*}
(L(\gt)f)(t) = \chi_\cI(t+\gt)f(t+\gt), \quad f \in L^p(\cI,\mathfrak{X}).
\end{equation*}
Using that we get
\bea
\lefteqn{
\left(L(\gt)\left(\e^{-\gt \cK} - \left(\e^{-\gt/n \cK_0}\e^{-\gt \cB/n}\right)^n\right)f\right)(t)}
\nonumber \\
&& =
\left\{U(t+\gt,t) - V_n(t+\gt,t)\right\} \, \chi_\cI(t+\gt)f(t) \ , \nonumber
\eea
for $\gt \ge 0$ and a.a. $t \in \cI$.

It turns out that for each $n \in \dN$ the operator
$L(\gt)\left(\e^{-\gt \cK} - \left(\e^{-\gt/n \cK_0}\e^{-\gt \cB/n}\right)^n\right)$ is a
multiplication operator
induced by $\{(U(t+\gt,t) - V_n(t+\gt,t)) \ \chi_\cI(t+\gt)\}_{t\in\cl I}$. As a consequence,
\bea
\lefteqn{\left\|L(\gt)\left(\e^{-\gt \cK} - \left(\e^{-\gt\cK_0/n}
\e^{-\gt \cB/n}\right)^n\right)\right\|_{\cL(L^p(\cI,\mathfrak{X}))}}
\nonumber \\
&& =
\esssup_{t\in \cI}\|U(t+\gt,t) - V_n(t+\gt,t)\|_{\cL(\mathfrak{X})} \ \chi_\cI(t+\gt) \ ,
\nonumber
\eea
for each $\gt \ge 0$. Note that one has
\bea
\lefteqn{
\sup_{\gt \ge 0}\left\|L(\gt)\left(\e^{-\gt \cK} -
\left(\e^{-\gt\cK_0/n}\e^{-\gt \cB/n}\right)^n\right)\right\|_{\cL(L^p(\cI,\mathfrak{X}))}}
\nonumber \\
&&= \esssup_{\gt \ge 0}\left\|L(\gt)\left(\e^{-\gt \cK} - \left(\e^{-\gt\cK_0/n}
\e^{-\gt \cB/n}\right)^n\right)\right\|_{\cL(L^p(\cI,\mathfrak{X}))} .
\nonumber
\eea
This is based on the fact that if $F(\cdot): \dR^+ \rightarrow \cL(\gotX)$ is strongly
continuous, then
$\sup_{\gt \ge 0}\|F(\gt)\|_{\cL(\gotX)} = \esssup_{\gt \ge 0}\|F(\gt)\|_{\cL(\gotX)}$.
Hence, we find
\bea
\lefteqn{
\sup_{\gt \ge 0}\left\|L(\gt)\left(\e^{-\gt \cK} -
\left(\e^{-\gt\cK_0/n}\e^{-\gt \cB/n}\right)^n\right)\right\|_{\cL(L^p(\cI,\mathfrak{X}))}}
\nonumber \\
&& =
\esssup_{\gt \ge 0}\esssup_{t\in \cI}\|U(t+\gt,t) -
V_n(t+\gt,t))\|_{\cL(\mathfrak{X})} \ \chi_\cI(t+\gt).\nonumber
\eea

Further, if $\Phi(\cdot,\cdot): \dR^+ \times \cI \rightarrow \cL(\mathfrak{X})$ is a strongly
measurable function, then
\begin{equation*}
\esssup_{(\gt,t)\, \in \, \dR^+\times \, \cI}\|\Phi(\gt,t)\|_{\cL(\mathfrak{X})} =
\esssup_{\gt \ge 0}\,\esssup_{t\in\cI}\|\Phi(\gt,t)\|_{\cL(\mathfrak{X})} .
\end{equation*}
Then, taking into account two last equalities, one obtains
\bea
\lefteqn{
\sup_{\gt \, \ge \, 0}\left\|L(\gt)\left(\e^{-\gt \cK} -
\left(\e^{-\gt\cK_0/n}\e^{-\gt \cB/n}\right)^n\right)\right\|_{\cL(\gotX)}} =
\nonumber \\
&& =
\esssup_{(\gt,t)\, \in \, \dR^+\times \, \cI}\|U(t+\gt,t) -
V_n(t+\gt,t)\|_{\cL(\mathfrak{X})} \ \chi_\cI(t+\gt)=\nonumber\\
&& =
\esssup_{(t,s)\, \in \, \gD}\|U(t,s) - V_n(t,s)\|_{\cL(\mathfrak{X})} \ , \nonumber
\eea
that proves (\ref{eq:2.0}).
\hfill $\square$
\end{proof}

We study bounded perturbations of the evolution generator  $\cK_0 = D_0$ (\ref{eq:1.6}).
To this aim we consider $\cl I =[0,1]$, $\mathfrak{X}= \C$ and we denote by  $L^p(\cl I)$ the
Banach space $L^p(\cl I, \C)$.

For $t \in \cl I$, let $q: t \mapsto q(t) \in L^\infty(\cl I)$.
Then, $q$ induces a \textit{bounded} multiplication operator $\cB = Q$ on the Banach space
$L^p(\cl I)$:
\begin{align}
(Qf)(t) = q(t) f(t), ~~f\in L^p(\cl I). \nonumber
\end{align}
For simplicity we assume that $q\geq 0$.
Then $Q$ generates on $L^p(\cl I)$ a contraction semigroup $\{\e^{- \tau Q}\}_{\tau \geq 0}$.
Since generator $Q$ is bounded, the closed operator $\cl A:= D_0 + Q$, with domain
$\dom(\cl A) = \dom(D_0)$, is generator of a $C_0$-semigroup on $L^p(\cl I)$. By the \textit{Trotter}
product formula in the \textit{strong} operator topology it follows immediately that
\begin{equation}\label{eq:3.00}
\lim_{n \rightarrow \infty}\left(\e^{-\gt D_0/n}\e^{-\gt Q/n}\right)^n f = \e^{-\gt (D_0+Q)}f,
\quad f\in L^p(\cI),
\end{equation}
uniformly in $\tau\in[0,T]$ on bounded time intervals.

Then we define on Banach space $\mathfrak{X} = \C$ a family of bounded operators
$\{V(t)\}_{t\in\cl I}$ by
\begin{align*}
 V(t) : =  \e^{- \int_0^t ds \, q(s) } \ , \quad t\in\cl I \, .
\end{align*}
Note that for almost every $t\in \cl I$ these operators are positive. Then $V^{-1}(t)$ exists
and it has the form
\begin{align*}
V^{-1}(t) =  \e^{ \int_0^t ds \, q(s) } \ , \quad t\in\cl I \, .
\end{align*}
The operator families $\{V(t)\}_{t\in\cl I}$ and $\{V^{-1}(t)\}_{t\in\cl I}$ induce two
bounded multiplication operators $\cl V $ and $\cl V ^{-1}$ on $L^p(\cl I)$, respectively. Then
invertibility implies that $\cl V \ \cl V^{-1} = \cl V^{-1} \, \cl V = \1 \big|_{L^p(\cl I)}\, $.
Using the operator $\cl V$ one easily verifies that $D_0+Q$ is {similar} to $D_0$, that is, one has
\begin{align}
\cl V^{-1}(D_0 + Q)\ \cl V= D_0 ~~\mathrm{or}~~D_0 + Q= \cl V \, D_0 \, \cl V^{-1} \ . \nonumber
\end{align}
Hence, the semigroup generated on $L^p(\cl I)$  by $D_0 + Q$ gets the explicit form:
\begin{align}\label{eq:3.0}
\left(\e^{-\tau(D_0 + Q)}f\right)(t) = \left(\cl V \ \e^{-\tau D_0} \, \cl V^{-1} f\right)(t) = \\
= \e^{-\int_{\, t-\tau}^t \, dy \, q(y) } \, f(t-\tau)\ \chi_{\cl I}(t-\tau) \nonumber \ .
\end{align}
Since by (\ref{eq:1.5}) the solution operator $U(t,s)$ that corresponds to evolution semigroup
(\ref{eq:3.0}) is defined by equation
\begin{align}
\left(\e^{-\tau(D_0 + Q)}\right)f(t) = U(t, t-\tau) f(t-\tau) \, \chi_{\cl I}(t-\tau) \ , \nonumber
\end{align}
we deduce that it is equal to $U(t, s) = \e^{-\int_s^t dy \, q(y) }$.

Now we study the corresponding Trotter product formula. For a fixed $\tau \geq 0$ and
$n\in \mathbb{N}$, we define \textit{approximates} $\{V_n\}_{n\geq 1}$ by
\begin{equation*}
\left(\left(\e^{- \tau D_0/n}\e^{- \tau Q/n}  \right)^n f\right)(t) =:
V_n(t,t-\gt) \, \chi_\cI(t-\gt)f(t-\gt) \ .
\end{equation*}
Then by straightforward calculations, which are similar to  (\ref{eq:2.01}), one finds that
approximants have the following explicit form:
\begin{equation*}
V_n(t,s)= \e^{- \tau_n \, \sum_{k=0}^{n-1} \, q\,(s + \, k \, \tau_n )} \ ,
\quad (t,s) \in \gD \ , \quad \tau_n = (t-s)/n \ , \quad n\in \mathbb{N} \, .
\end{equation*}
%
\begin{theorem}\label{th:2.2}
Let $q \in L^\infty(\cI)$ be non-negative. Then
\begin{equation*}
\begin{split}
&\sup_{\gt \ge 0}\left\|\e^{-\gt(D_0 + Q)} -
\left(\e^{-\gt D_0/n}\e^{-\gt Q/n}\right)^n\right\|_{\cL (L^p(\cl I))} =
\\
&
\gT\left(\esssup_{(t,s)\in\gD}\Big|\int^t_s \, dy \, q(y) - \tau_n \, \sum_{k=0}^{n-1} q \, (s +
k \, \tau_n)\Big|\right) \ , \quad n\in \mathbb{N} \, ,
\end{split}
\end{equation*}
as $n\to\infty$, where $\gT$ is the \textit{Landau} symbol defined in Remark \emph{\ref{rem:8.3.1}},
see Subsection \emph{\ref{subsec:8.3.1}}.
\end{theorem}
\begin{proof}
First, by Theorem \ref{th:2.1} and by $U(t, s) = \e^{-\int_s^t dy \, q(y) }$ we obtain
\bea \label{eq:3.01}
\lefteqn{\sup_{\gt \ge 0}\left\|\e^{-\gt(D_0 + Q)} -
\left(\e^{-\gt D_0/n}\e^{-\gt Q/n}\right)^n\right\|_{\cL (L^p(\cl I))}}\\
& &= \esssup_{(t,s)\in \gD}\left|\, \e^{-\int^t_s dy \, q(y)} - \e^{- \tau_n \, \sum_{k=0}^{n-1}
q(s +  \, k \, \tau_n)}\right|\nonumber \ .
\eea
Then, using the inequality
\begin{equation*}
\e^{-\max\{x,y\}}\, |x-y| \le |\, \e^{-x} - \e^{-y}| \le |x-y|, \quad 0 \le x, y \ ,
\end{equation*}
for $0 \leq s < t \leq 1$ one finds the estimates
\bed
\e^{-\|q\|_{L^\infty}} \, R_n(t,s;q) \le \\
\Big|\, \e^{-\int^t_s dy \, q(y)} - \e^{- \tau_n \,
\sum_{k=0}^{n-1} q(s + \, k\, \tau_n)}\Big| \le R_n(t,s;q) \ ,
\eed
where
\begin{equation}\label{eq:3.02}
R_n(t,s,q) := \Big|\int^t_s dy \, q(y) - \tau_n \, \sum_{k=0}^{n-1} q(s + k\, \tau_n) \, \Big| \ ,
\quad (t,s) \in \gD \ , \quad n\in \mathbb{N} \, .
\end{equation}
Hence, for the left-hand side of (\ref{eq:3.01}) we get the estimate
\begin{equation*}
\e^{-\|q\|_{L^\infty}} \ R_n(q) \le \sup_{\gt \ge 0}\left\|\e^{-\gt(D_0 + Q)} - \left(\e^{-\gt D_0/n}
\e^{-\gt Q/n}\right)^n\right\|_{\cL (L^p(\cl I))} \le R_n(q) \ ,
\end{equation*}
where $R_n(q) := \esssup_{(t,s)\in \gD}R_n(t,s;q)$, $n \in \dN$. These estimates together with
definition of $\gT$ prove the assertion.
\hfill $\square$
\end{proof}

Note that by virtue of (\ref{eq:3.02}) and Theorem \ref{th:2.2} the operator-norm
convergence rate of the Trotter product formula for the pair $\{D_0 , Q\}$ coincides with the
convergence rate of the integral \textit{Darboux-Riemann} sum approximation of the \textit{Lebesgue}
integral.

\subsection{\blue{Rate of convergence}} \label{subsec:8.3.3}
First we consider the case of a real H\"older-continuous function $q\in C^{0,\gb}(\cI)$.
\begin{theorem}\label{th:3.1}
If {{$q \in C^{0,\gb}(\cI)$} is non-negative}, then
\begin{equation*}
\sup_{\gt \ge 0}\left\|e^{-\gt \,(D_0 + Q)} -
\left(e^{-\gt \, D_0/n}\, e^{-\gt \, Q/n}\right)^n\right\|_{\cl L(L^p(\cl I))} = O({1}/{n^\gb}) \ ,
\end{equation*}
as $n \to \infty$.
\end{theorem}
\begin{proof}
One gets for $\tau_n = (t-s)/n$ and $n\in \mathbb{N}$ :
\begin{equation*}
\begin{split}
\int^t_s dy \, q(y) &- \tau_n \, \sum^{n-1}_k q \, (s + {k} \, \tau_n )\\
=& \sum^{n-1}_{k=0}\int^{(k+1)\tau_n}_{{k} \, \tau_n} dy \, \left(q(s+y) -
q(s+ {k}\, \tau_n)\right)\ ,
\end{split}
\end{equation*}
which yields the estimate
\begin{equation*}
\begin{split}
\Big|\int^t_s dy \, q(y) &- \tau_n \, \sum^{n-1}_k q \, (s + {k}\tau_n)\Big|\\
\le& \sum^{n-1}_{k=0}\int^{(k+1)\tau_n}_{{k}\tau_n}
dy \, \left|\, q(s+y) - q(s + {k} \, \tau_n)\right| \ .
\end{split}
\end{equation*}
Since $q \in C^{0,\gb}(\cI)$,  there is a constant $L_\gb > 0$ such that for $y\in[k \tau_n,
(k+1) \tau_n]$ one has
\begin{equation*}
\left|q(s+y) - q(s+ {k}\tau_n) \right| \le L_\gb \ |y- {k}\tau_n |^\gb \le
L_\gb \ \frac{(t-s)^\gb}{n^\gb}\ .
\end{equation*}
Hence, we find
\begin{equation*}
\Big|\int^t_s dy \, q(y) - \tau_n \, \sum^{n-1}_k q(s + {k} \tau_n)\Big|
\le L_\gb \ \frac{(t-s)^{1+\gb}}{n^\gb} \le L_\gb \ \frac{1}{n^\gb} \ ,
\end{equation*}
which proves (cf. Remark \ref{rem:8.3.1})
\begin{equation*}
\esssup_{(t,s)\in \gD}\Big|\int^t_s dy \, q(y) - \tau_n \sum^{n-1}_k q(s + {k}\, \tau_n)\Big|
= O\left(\frac{1}{n^\gb}\right)\, .
\end{equation*}
Applying now Theorem \ref{th:2.2} one completes the proof.
\hfill $\square$
\end{proof}

The next natural question is: what happens, when function $q$ is only \textit{continuous}?
%
\begin{theorem}\label{th:3.2}
If $q: \cl I \rightarrow \C$ is continuous and non-negative, then
\begin{equation}\label{eq:4.10}
\left\|e^{-\gt(D_0 + Q)} - \left(e^{-\gt D_0/n}e^{-\gt Q/n}\right)^n\right\|_{\cl L( L^p(\cl I))}
= o(1) \ ,
\end{equation}
as $n\to \infty$.
\end{theorem}
\begin{proof}
Seeing that $q$ is continuous, for any $\varepsilon > 0$ there is $\gd > 0$ such that for
$y,x \in \cI$ and $|y-x| < \gd$ we have $|q(y) - q(x)| < \varepsilon$.
Therefore, if $1/n < \gd$, then for $y \in ({k} \tau_n, (k+1)\tau_n)$ we have
\begin{equation*}
|q(s+y) - q(s+ {k}\, \tau_n)| < \varepsilon, \quad (t,s) \in \gD \ .
\end{equation*}
Hence,
\begin{equation*}
\Big|\int^t_s dy \, q(y) - \tau_n \sum^{n-1}_k q(s + {k}\,\tau_n)\Big|
\le \varepsilon (t-s) \le \varepsilon \ ,
\end{equation*}
which yields (cf. Remark \ref{rem:8.3.1})
\begin{equation*}
\esssup_{(t,s)\in\gD}\Big|\int^t_s  dy \, q(y) - \tau_n \, \sum^{n-1}_k q(s + {k}\,\tau_n)\Big|
= o(1) \ .
\end{equation*}
Now it remains only to apply Theorem \ref{th:2.2}.
\hfill $\square$
\end{proof}

{\blue{Here it is worth to note that for \textit{general} continuous function $q$ one can say
\textit{nothing} about the convergence rate. Indeed, it can be shown that in (\ref{eq:4.10}) the
convergence to zero can be \textit{arbitrary slow} (\ref{eq:4.11}) for a bounded perturbation $Q$.
This is drastically different to the case, when dominating generator corresponds to a holomorphic
semigroup and perturbation operator is bounded, cf. (\ref{eq:8.2.17}), (\ref{eq:8.2.19}) for
$\alpha = 0 \,$, or to the case of unbounded perturbation, when $0 < \alpha < 1$, see
(\ref{eq:8.2.16}), (\ref{eq:8.2.18}).}}
\begin{theorem}\label{th: 4.3}
Let $\gd_n>0$ be a sequence with $\gd_n \to 0$ as $n \to \infty$. Then there exists a continuous
function $q:\cI = [0,1] \rightarrow \dR$, such that
\be\label{eq:4.11}
\sup_{\gt \ge 0}\left\|e^{-\gt(D_0 + Q)} - \left(e^{-\gt D_0/n}
e^{-\gt Q/n}\right)^n\right\|_{\cl L( L^p(\cl I))} = \go(\gd_n) \, ,
\ee
as $n\to\infty$, where $\omega$ is the \textit{Landau} symbol defined in Remark \emph{\ref{rem:8.3.1}},
see Subsection \emph{\ref{subsec:8.3.1}}.
\end{theorem}
\begin{proof}
Taking into account the {\blue{\textit{Walsh-Sewell} theorem}} (\cite{WaSe37}, Theorem 6),
we find that for any sequence $\{\gd_n\}_{n\in\dN}$, $\gd_n > 0$ satisfying
$\lim_{n\to\infty}\gd_n = 0$ there exists a continuous function
$f: [0,2\pi] \longrightarrow \dR$, such that
\begin{equation*}
\left|\int^{2\pi}_0  dx \, f(x)  - \frac{2\pi}{n}\sum^n_{k=1}f(2k\pi/n)\right| = \go(\gd_n) \ ,
\end{equation*}
as $n \to \infty$. Setting $q(y) := f(2\pi(1-y))$ for $y \in [0,1]$, we get a continuous function
$q: [0,1] \longrightarrow \dR$, such that
\begin{equation*}
\left|\int^{1}_0 dy \, q(y)- \frac{1}{n}\sum^{n-1}_{k=0}q(k/n)\right| = \go(\gd_n) \ .
\end{equation*}
Given that function $q$ is continuous, we find
\begin{equation*}
\begin{split}
\esssup_{(t,s)\in \gD}\Big|\int^t_s dy \, q(y)&- \tau_n \, \sum^{n-1}_{n=0} q(s +
k \, \tau_n)\Big|\\
& \ge \Big|\int^{1}_0 dy \, q(y) - \frac{1}{n}\sum^{n-1}_{k=0}q(k/n)\Big| \ ,
\end{split}
\end{equation*}
which yields
\begin{equation*}
\esssup_{(t,s)\in \gD}\Big|\int^t_s dy \, q(y) - \tau_n \sum^{n-1}_{n=0}q(s +
k \, \tau_n)\Big| = \go(\gd_n) \ .
\end{equation*}
Applying now Theorem \ref{th:2.2} we prove \eqref{eq:4.11}.
\hfill $\square$
\end{proof}

Our final comment concerns the case, when function $q:[0,1] \longrightarrow \dR$ is only
\textit{measurable}. Then it can happen that the Trotter product formula for that pair
$\{D_0 , Q\}$ \textit{does not} converge in the \textit{operator-norm} topology.

\begin{theorem}\label{th:4.4}
 There is a non-negative function $q \in L^\infty([0,1])$ such that
 \be\label{eq:4.16}
 \limsup_{n\to\infty}\; \sup_{\gt \ge 0}\left\|e^{-\gt(D_0 + Q)} - \left(e^{-\gt D_0/n}
 e^{-\gt Q/n}\right)^n\right\|_{\cl L( L^p(\cl I))} > 0 \ .
 \ee
\end{theorem}
\begin{proof}
Let us introduce open intervals:
\begin{equation*}
\begin{split}
 \gD_{0,n} &:= (0,\, \tfrac{1}{2^{2n+2}}),\\
 \gD_{k,n} &:= (t_{k,n} - \tfrac{1}{2^{2n +2}} \, , \ t_{k,n} + \tfrac{1}{2^{2n +2}}),
 \quad k = 1,2,\ldots,2^n-1,\\
 \gD_{2^n,n}&:= (1-\tfrac{1}{2^{2n+2}}, \, 1),
\end{split}
\end{equation*}
$n \in \dN$, where
\begin{equation*}
t_{k,n} = \frac{k}{2^n}, \quad k = 0,\ldots, n,\quad n \in \dN.
\end{equation*}
Notice that $t_{0,n} = 0$ and $t_{2^n,n} = 1$. One easily checks that the intervals
$\gD_{k,n}$, $k = 0,\ldots,2^n$, are mutually disjoint. We introduce the open sets
\begin{equation*}
\cO_n = \bigcup^{2^n}_{k = 0} \gD_{k,n} \subseteq \cI, \quad n \in \dN.
\end{equation*}
and
\begin{equation*}
\cO = \bigcup_{n\in\dN}\cO_n \subseteq \cI.
\end{equation*}
Then it is clear that
\begin{equation*}
|\cO_n| = \frac{1}{2^{n+1}}, \quad n \in \dN,
\quad \mbox{and} \quad
|\cO| \le \frac{1}{2}.
\end{equation*}
Therefore, the \textit{Lebesgue} measure of the closed set $\cC := \cI\setminus\cO \subseteq \cI$
can be estimated by
\begin{equation*}
|\cC| \ge \frac{1}{2} \ .
\end{equation*}
Using the characteristic function $\chi_{\cC}(\cdot)$ of the set $\cC$ we define
\begin{equation*}
q(t) :=\chi_{\cC}(t), \quad t \in \cI \ .
\end{equation*}
The function $q$ is measurable and it satisfies $0 \le q(t) \le 1$, $t \in \cI$.

Let $\varepsilon \in (0,1)$. We choose $s \in (0,\varepsilon)$ and $t \in (1-\varepsilon,1)$ and
we set
\begin{equation*}
\xi_{k,n}(t,s) := s + k \ \frac{t-s}{2^n}, \quad k = 0,\ldots,2^n-1, \quad n \in \dN,
\quad (t,s) \in \gD.
\end{equation*}
Note that $\xi_{k,n}(t,s) \in (0,1)$, $k = 0,\ldots,2^n-1$, $n \in \dN$. Moreover, we have
\begin{equation*}
t_{k,n} - \xi_{k,n}(t,s) =  k \ \frac{1}{2^n} -s - k \ \frac{t-s}{2^n} = k \ \frac{1-t+s}{2^n} -s \ ,
\end{equation*}
which leads to the estimate
\begin{equation*}
|t_{k,n} - \xi_{k,n}(t,s)| \le \varepsilon \ \,  (1 + {k}/{2^{n-1}}), \quad k =0,\ldots,2^n-1,
\quad n \in \dN \ .
\end{equation*}
Hence
\begin{equation*}
|t_{k,n} - \xi_{k,n}(t,s)| \le 3\varepsilon, \quad k =0,\ldots,2^n-1, \quad n \in \dN.
\end{equation*}
Let $\varepsilon_n := {1}/{(3\cdot 2^{2n+2})}$ for $n \in \dN$. Then we get that
$\xi_{k,n}(t,s) \in \gD_{k,n}$
for $k = 0,\ldots,2^n-1$, $n \in \dN$, $s \in (0,\varepsilon_n)$ and for $t \in (1-\varepsilon_n,1)$.

Now let
\begin{equation*}
S_n(t,s;q) := \tau_n \sum^{n-1}_{k=0}q \, (s + k \, \tau_n) \, ,
\quad (t,s) \in \gD \, , \quad \tau_n = (t-s)/n \, ,  \quad n \in \dN \, .
\end{equation*}
We consider
\begin{equation*}
S_{2^n}(t,s;q) = \frac{t-s}{2^n} \ \sum^{2^n-1}_{k=0}q \, (s + k \, \tfrac{t-s}{2^n}) =
\frac{t-s}{2^n} \ \sum^{2^n-1}_{k=0}q \,(\xi_{k,n}(t,s)) \, ,
\end{equation*}
$n \in \dN$, $(t,s) \in \gD$. If $s \in (0,\varepsilon_n)$ and $t \in (1-\varepsilon_n,1)$,
then $S_{2^n}(t,s;q) = 0$, $n \in \dN$ and
\begin{equation*}
\left|\int^t_s dy \, q(y)  - S_{2^n}(t,s;q)\right| = \int^t_s dy \, q(y) \ , \quad n \in \dN \, ,
\end{equation*}
for $s \in (0,\varepsilon_n)$ and $t \in (1-\varepsilon_n,1)$. In particular, this yields
\bed
\esssup_{(t,s)\in \gD}\left|\int^t_s dy \, q(y)  - S_{2^n}(t,s;q)\right| \ge \esssup_{(t,s)
\in\gD}\int^t_s dy \, q(y) \ge \int_{\cI} dy \, \chi_\cC(y) \ge \frac{1}{2} \ .
\eed
Hence, we obtain
\begin{equation*}
\limsup_{n\to\infty}\;\esssup_{(t,s)\in \gD}\left|\int^t_s dy \, q(y)  - S_{2^n}(t,s;q)\right|
\ge \frac{1}{2},
\end{equation*}
and applying Theorem \ref{th:2.2} we finish the prove of \eqref{eq:4.16}.
\end{proof}
\begin{remark}\label{rem:8.3.2}
\rm{
We note that Theorem \ref{th:4.4} does not exclude the convergence of the Trotter product
formula for the pair $\{D_0 , Q\}$ in the \textit{strong} operator topology.
We would remind that the same kind of \textit{dichotomy} is known for the Trotter product formula
on a Hilbert space, see the \textit{Hiroshi Tamura} example in \cite{Tam00}, Theorem B.
By virtue of (\ref{eq:3.00}) and (\ref{eq:4.16}), Theorem \ref{th:4.4} yields an example of
this {dichotomy} on a Banach space.

Again, there is a drastic difference between the origin of these conclusions in a Hilbert space
(\cite{Tam00}, Theorem B) for \textit{unbounded} perturbation of the \textit{holomorphic} semigroup
and in a Banach space (Theorem \ref{th:4.4}) for \textit{bounded} perturbation of a
(non-holomorphic) \textit{contractive} semigroup.
}
\end{remark}
\section{{\textbf{Notes}}} \label{sec:8.4}
{Notes to Section~\ref{sec:8.1}}.
Characterisation of \textit{holomorphic contraction} semigroups at the end of
Subsection \ref{subsec:8.1.2} (iv), is due to Corollary II.4.9 \cite{EN00}.
For the proof of Lemma \ref{lem:8.1.11} see, for example \cite{Tan75}, Lemma 2.3.5.

\bigskip
\noindent
{Notes to Section~\ref{sec:8.2}}.
Here we extend to the \textit{operator-norm} convergence
of the product formula on a Banach space for perturbation
$B$ with a relative \textit{zero} $A$-bound for holomorphic semigroup $\{\e^{-tA}\}_{t\geq 0}$
some of the Trotter-Chernoff results, cf. \cite{Trot59}, \cite{But20}, \cite{Zag20}.
This shows that hypothesis of self-adjointness in the case of a Hilbert space \cite{IT97}
has only a technical importance.

On the other hand the operator-norm topology is "natural" for holomorphic $C_0$-semigroups, which
may lead one to think that it is an essential hypothesis for the operator-norm convergence
of the Trotter product formula. In Section \ref{sec:8.3} we showed that this hypothesis is also
technical, but we have to assume
contraction of semigroup $\{\e^{-tA}\}_{t\geq 0}$. We would like to remark that demand of
contraction is not as superfluous as one could suppose. For demonstration we address the reader
to instructive example by Trotter \cite{Trot59}, where it is called "the norm condition".

This section contains a revision of result \cite{CZ01}, Section 3,
where the \textit{operator-norm} convergence of the Trotter product formula on a Banach space
$\mathfrak{X}$ has been proven (up to our knowledge) for the first time.
For a survey of similar results in this direction see \cite{NSZ18b}.


\bigskip
\noindent
{Notes to Section~\ref{sec:8.3}}.
In contrast to Section \ref{sec:8.2}, where \textit{operator-norm} convergence holds true if the
dominating operator $A \in \mathscr{H}_{c}(\theta, 0)$ generates a \textit{holomorphic}
contraction semigroup and operator $B$ is a
$A$-infinitesimally small generator of a \textit{contraction} semigroup (in particular,
if $B$ is a \textit{bounded} operator), we present \textit{Example} that this is also possible if
condition on generator $A$ is \textit{relaxed}. The conditions are \cite{NSZ18a}: \\
(1) Operator $A = \cK_0$ generates a \textit{contractive} (not holomorphic!) semigroup.\\
(2) $B = \cB$ is a \textit{bounded} operator.

There it is also demonstrated that the operator-norm convergence generally \textit{fails}
(even for bounded operators $\cB$) if \textit{unbounded} $\cK_0$ is \textit{not} a \textit{holomorphic}
generator and that operator-norm convergence of the Trotter product formula can be
\textit{arbitrary slow}, cf. {Subsection} \ref{subsec:8.3.3}. This is again very different to
the holomorphic case: $A \in \mathscr{H}_{c}(\theta, 0)$ (cf. Subsection \ref{subsec:8.2.2}),
where the rate of the operator-norm convergence is of the order $O({(\ln n)^2}/{n})$ for
\textit{any} bounded perturbation $B$ ($\alpha =0$), see Theorems \ref{th:8.2.6}, \ref{th:8.2.7},
and Corollary \ref{cor:8.2.8}.



\end{document}